\documentclass[a4paper, 12pt]{amsart}
\usepackage[all]{xy}
\usepackage{hyperref}
\usepackage{amssymb, color, amsmath, comment, mathrsfs}
\usepackage[T1]{fontenc}
\usepackage{tikz, epsfig}
\usepackage[normalem]{ulem}
\usepackage{cancel}

\numberwithin{equation}{section}

\theoremstyle{theorem}
\newtheorem{thm}{Theorem}[section]
\newtheorem{lemma}[thm]{Lemma}
\newtheorem{prop}[thm]{Proposition}
\newtheorem{cor}[thm]{Corollary}

\theoremstyle{definition}
\newtheorem{defn}[thm]{Definition}

\theoremstyle{remark}
\newtheorem{eg}[thm]{Example}
\newtheorem{rmk}[thm]{Remark}

\DeclareMathOperator{\im}{im}

\newcommand{\id}{\operatorname{id}}

\newcommand{\hull}{\operatorname{hull}}

\newcommand{\Prim}{\operatorname{Prim}}
\newcommand{\Per}{\operatorname{Per}}

\newcommand{\Ind}{\operatorname{Ind}}
\newcommand{\Hom}{\operatorname{Hom}}
\newcommand{\ospan}{\overline{\operatorname{span}}}
\newcommand{\espan}{\operatorname{span}}
\newcommand{\lt}{\operatorname{lt}}
\newcommand{\qH}{j^H}
\newcommand{\qHa}[1]{j^H_{#1}}
\newcommand{\ext}[1]{i^H(#1)}

\newcommand{\NN}{\mathbb{N}}
\newcommand{\TT}{\mathbb{T}}
\newcommand{\ZZ}{\mathbb{Z}}

\newcommand{\BB}{\mathbb{B}}
\newcommand{\FF}{\mathbb{F}}

\newcommand{\Gg}{\mathcal{G}}
\newcommand{\Kk}{\mathcal{K}}

\newcommand{\minimatrix}[4]{\bigl(\begin{smallmatrix}#1 & #2 \\ #3 & #4\end{smallmatrix} \bigr)}
\newcommand{\minivector}[2]{\bigl(\begin{smallmatrix}#1 \\ #2 \end{smallmatrix} \bigr)}

\newcommand{\Etilde}{\tilde{E}}
\newcommand{\dtilde}{\tilde{d}}
\newcommand{\Kktilde}{\tilde{\Kk}}
\newcommand{\alphatilde}{\tilde{\alpha}}

\title[Real rank and topological dimension of higher rank graph algebras]{Real rank and topological dimension of higher rank graph algebras}

\date{\today}

\author{David Pask}
\email{dpask/asierako/asims@uow.edu.au}
\author{Adam Sierakowski}
\author{Aidan Sims}
\address[D. Pask, A. Sierakowski and A. Sims]{School of Mathematics and Applied Statistics \\
University of Wollongong\\
Wollongong NSW 2522\\
AUSTRALIA}

\subjclass[2010]{46L05}
\keywords{Graph $C^*$-algebra; real rank; topological dimension; purely infinite; higher-rank graph}
\thanks{This research was supported by the Australian Research Council.}

\begin{document}

\begin{abstract}
We study dimension theory for the $C^*$-algebras of row-finite $k$-graphs with no
sources. We establish that strong aperiodicity---the higher-rank analogue of
condition~(K)---for a $k$-graph is necessary and sufficient for the associated
$C^*$-algebra to have topological dimension zero. We prove that a purely infinite
$2$-graph algebra has real-rank zero if and only if it has topological dimension zero and
satisfies a homological condition that can be characterised in terms of the adjacency
matrices of the $2$-graph. We also show that a $k$-graph $C^*$-algebra with topological
dimension zero is purely infinite if and only if all the vertex projections are properly
infinite. We show by example that there are strongly purely infinite $2$-graphs algebras,
both with and without topological dimension zero, that fail to have real-rank zero.
\end{abstract}

\maketitle

\section{Introduction}
Dimension theory of $C^*$-algebras has played an important role in the subject over the
last few decades. In particular, notions of dimension such as real rank, stable rank,
decomposition rank and nuclear dimension have become more and more prominent in
classification theory. But these notions of dimension can be difficult to compute for
concrete classes of examples. In this paper we focus on dimension theory of higher rank
graph $C^*$-algebras. Specifically we consider topological dimension zero and real-rank
zero of $C^*$-algebras associated to higher rank graphs.

A $C^*$-algebra $A$ has \emph{topological dimension zero} if the primitive ideal space of
$A$ endowed with the Jacobsen topology has a basis of compact open sets. This definition
can be found in Brown and Pedersen's work \cite{BroPed2}, but the study of the concept
goes back much further. For example Bratteli and Elliott showed \cite{BraEll} that a
separable $C^*$-algebra $A$ of type I is AF if and only if it has topological dimension
zero. For more recent work on topological dimension see \cite{Bro, BroPed2, KanPas, Pas2,
PasPhi, PasRor, RobTik, Thi}.

Recently, sufficient conditions on a $k$-graph $\Lambda$ for its $C^*$-algebra to have
topological dimension zero were established by Kang and the first named author. They
proved in \cite[Theorem~4.2]{KanPas} that for a row-finite $k$-graph $\Lambda$ with no
sources, the $C^*$-algebra associated to $\Lambda$ has topological dimension zero if
$\Lambda$ is strongly aperiodic. Strong aperiodicity is the higher-rank analogue of the
well-known condition~(K) for directed graphs (every vertex that is the basepoint of a
first-return path is the basepoint of at least two such paths). It can be characterised
combinatorially in terms of the $k$-graph, and can be rephrased in a number of other
ways. For example, it is equivalent to the condition on the infinite-path groupoid
$\Gg_\Lambda$ of $\Lambda$, that for every closed invariant subset of the unit space, the
points with trivial isotropy are dense. It is also equivalent to the property that every
ideal of the $k$-graph $C^*$-algebra $C^*(\Lambda)$ is gauge-invariant. We prove in
Section~\ref{sec:topdim} that this sufficient condition is also necessary
(Theorem~\ref{thm2.3}). Our approach uses the fact that topological dimension zero passes
to ideals and quotients, together with recent results on $P$-graphs, induced algebras,
and primitive-ideal structure of $k$-graph algebras \cite{CarKanShoSim, Wil}. We finish
the section with Corollary~\ref{thm.ip} which describes a number of properties of
$\Lambda$ and $C^*(\Lambda)$ that are all equivalent to strong aperiodicity of $\Lambda$,
and hence topological dimension zero for $C^*(\Lambda)$.

We then consider real-rank zero for purely infinite $k$-graph $C^*$-algebras. A
$C^*$-algebra $A$ has \emph{real-rank zero} if the set of invertible self-adjoint
elements of the minimal unitisation of $A$ is dense in the set of all self-adjoint
elements in the minimal unitisation of $A$ \cite{BroPed}. For $k=1$, it was established
in the locally finite case in \cite{JeoParShi, JeoPar}, and in general by Hong and
Szyma\'nski in \cite[Theorem~2.5]{HonSzy}, that a graph $C^*$-algebra $C^*(E)$ has
real-rank zero if and only if the graph satisfies Condition~(K). However, the methods of
\cite{JeoParShi, JeoPar, HonSzy} do not generalise to $k$-graphs. So we restrict our
attention to the special case of $k$-graphs $\Lambda$ for which $C^*(\Lambda)$ is purely
infinite.

Simple purely infinite $C^*$-algebras are automatically of real-rank zero by
\cite{BroPed, Zha}, but non-simple purely infinite $C^*$-algebras need not have real-rank
zero (see Examples \ref{eq6.3}~and~\ref{eq6.4}). Indeed, Pasnicu and R{\o}rdam have
proved that a purely infinite $C^*$-algebra has real-rank zero if and only if it has
topological dimension zero and satisfies a $K$-theoretic criterion called
$K_0$-liftability \cite{PasRor}. Using this result we derive a homological condition
characterising when a purely infinite $2$-graph $C^*$-algebra has real-rank zero. We do
this by examining Evans' spectral-sequence calculation of $K$-theory for $k$-graph
$C^*$-algebras. This sequence is based on a complex $D^\Lambda_*$ in which the terms are
the exterior powers of $\ZZ^k$ tensored with $\ZZ\Lambda^0$. When $k = 2$, Evans proves
that $K_1(C^*(\Lambda))$ is isomorphic to the first homology group of this complex. We
show that the inclusion $j^H : H \hookrightarrow \Lambda^0$ of a saturated hereditary set
$H$ induces a morphism from the spectral sequence for $\Gamma := H\Lambda$ to the
spectral sequence for $\Lambda$. Using naturality of Kasparov's spectral sequence, we
show that for $k = 2$ the map in $K_1$ induced by the inclusion $C^*(\Gamma)
\hookrightarrow C^*(\Lambda)$ coincides with the map $H_1(j^H) := H_1(D^\Gamma_*) \to
H_1(D^\Gamma_*)$ induced by the inclusion $H \hookrightarrow \Lambda^0$. Combining this
with Pasnicu and R{\o}rdam's result yields our characterisation of real-rank zero: a
purely infinite $2$-graph $C^*$-algebra has real rank zero if and only if $\Lambda$ is
strongly aperiodic and $H_1(j^H)$ is injective for each saturated hereditary subset $H$
of $\Lambda^0$ (Theorem~\ref{thm4.2}). Both of these conditions are necessary for
$C^*(\Lambda)$ to have real rank zero even when $C^*(\Lambda)$ is not purely infinite.
Moreover, the injectivity of each $H_1(j^H)$ can be reformulated as an elementary
algebraic condition involving the adjacency matrices of the ambient $2$-graph.

Our resulting condition on $\Lambda$ is much more subtle than the corresponding result
for $1$-graphs, reflecting the more nuanced $K$-theory of $k$-graph $C^*$-algebras in
which, for example, the $K_0$-classes of vertex projections do not necessarily generate
the whole $K_0$-group (see \cite{Eva}).

Since pure infiniteness is a hypothesis in Theorem~\ref{thm4.2}, we develop in
Section~\ref{sec:pi} a characterisation of pure infiniteness for the $C^*$-algebras of
strongly aperiodic $k$-graphs. This topic has been intensively studied, and for 1-graphs
this culminated in \cite{HonSzy}, where sufficient and necessary conditions for pure
infiniteness of a $1$-graph $C^*$-algebra (in terms of the $1$-graph) where established.
Building on recent results for cofinal $k$-graphs in \cite{BroClaSie}, we establish a
number of properties equivalent to pure infiniteness of $C^*(\Lambda)$. In particular,
generalising \cite[Corollary~5.1]{BroClaSie}, we prove that $C^*(\Lambda)$ is purely
infinite if and only if the vertex projections $\{s_v : v\in\Lambda^{0} \}$ are all
properly infinite (without assuming $\Lambda$ is cofinal). We also show that this is
equivalent to asking that for any saturated hereditary $H \subseteq \Lambda^0$, the
vertex projections of $C^*(\Lambda \setminus \Lambda H)$ are all infinite (not, \emph{a
priori}, properly infinite).

Using this, we prove that if a $k$-graph has an aperiodic quartet \cite{KanPas} at every
vertex, then its $C^*$-algebra is purely infinite (Proposition~\ref{prop:aq-pis}). Our
result adds to the list (see \cite{EvaSim,KumPas,Sim}) of known sufficient conditions on
a $k$-graph $\Lambda$ for pure infiniteness of its $C^*$-algebra. The utility of our
result becomes apparent in Section~\ref{sec:examples}, where we use
Proposition~\ref{prop:aq-pis} to construct two examples of purely infinite $k$-graph
$C^*$-algebras without real-rank zero. The first example is a $2$-graph $\Lambda$ such
that $C^*(\Lambda)$ is purely infinite and has topological dimension zero, but fails to
have real-rank zero because the morphism $H_1(\qH)$, for a suitable $H \subseteq
\Lambda^0$, is not injective. Our second example is also purely infinite, but fails to
have real-rank zero because the graph is not strongly aperiodic, and so its $C^*$-algebra
does not have topological dimension zero. Finally we present an example of a $2$-graph
which is both strongly aperiodic and satisfies the injectivity condition, but whose
$C^*$-algebra is not purely infinite, and fails to have real rank zero.

\section{Preliminaries and notation}\label{sec:prelims}
We let $\Prim(A)$ denote the set of all primitive ideals in a $C^*$-algebra $A$. The
Jacobsen topology on $\Prim(A)$ is defined as follows. For each ideal $I$ of $A$, we
define the hull of $I$ by $\hull(I) := \{P\in \Prim(A): I\subseteq P\}$; and for each
subset $S \subseteq \Prim(A)$, we define the kernel of $S$ to be $\ker(S) := \bigcap_{P
\in S} P$. The Jacobsen topology is then determined by the closure operation:
$\overline{S}:=\hull(\ker(S))$ for every $S \subseteq \Prim(A)$.

Following \cite{KumPas, PasQuiRae, RaeSimYee} we briefly recall the notion of a $k$-graph
and the associated notation. For $k \ge 0$, a \emph{$k$-graph} is a nonempty countable
small category equipped with a functor $d : \Lambda \to \NN^k$ that satisfies the
\emph{factorisation property}: for all $\lambda \in \Lambda$ and $m,n\in \NN^k$ such that
$d(\lambda) = m + n$ there exist unique $\mu, \nu\in \Lambda$ such that $d(\mu) = m$,
$d(\nu) = n$, and $\lambda = \mu\nu$. When $d(\lambda) = n$ we say $\lambda$ has
\emph{degree} $n$, and we write $\Lambda^n = d^{-1}(n)$. The standard generators of
$\NN^k$ are denoted $e_1, \dots ,e_k$, and we write $n_i$ for the $i^{\textrm{th}}$
coordinate of $n \in \NN^k$. For $m,n\in \NN^k$, we write $m \vee n $ for their
coordinate-wise maximum, and define a partial order on $\NN^k$ by $m\leq n$ if $m_i\leq
n_i$ for all $i$.

If $\Lambda$ is a $k$-graph, its \emph{vertices} are the elements of $\Lambda^0$. The
factorisation property implies that these are precisely the identity morphisms, and so
can be identified with the objects. For $\alpha \in \Lambda$ the \emph{source}
$s(\alpha)$ is the domain of $\alpha$, and the \emph{range} $r(\alpha)$ is the codomain
of $\alpha$ (strictly speaking, $s(\alpha)$ and $r(\alpha)$ are the identity morphisms
associated to the domain and codomain of $\alpha$). Given $\lambda,\mu\in \Lambda$ and
$E\subseteq \Lambda$, we define
\begin{gather*}
\lambda E=\{\lambda\nu : \nu\in E, r(\nu)=s(\lambda)\},\quad E\mu=\{\nu\mu : \nu\in E, s(\nu)=r(\mu)\},\quad\text{and}\\
\lambda E\mu=\lambda E\cap E\mu.
\end{gather*}
For $X,E,Y\subseteq\Lambda$, we write $XEY$ for $\bigcup_{\lambda\in X, \mu\in Y}\lambda
E\mu$. We say the $k$-graph $\Lambda$ is \textit{row-finite} if the set $v{\Lambda}^{m}$
is finite for each $m\in\mathbb{N}^k$ and $v\in{\Lambda}^0$. Also, $\Lambda$ has
\textit{no sources} if $v{\Lambda}^{e_i}\ne \emptyset$ for all $v\in{\Lambda}^0$ and
$i\in\{1,\dots,k\}$.

For $\lambda\in\Lambda$ and $0\leq m\leq n\leq d(\lambda)$, we write $\lambda(m,n)$ for
the unique path in $\Lambda$ such that $\lambda=\lambda'\lambda(m,n)\lambda''$, where
$d(\lambda') = m$, $d(\lambda(m,n)) = n-m$ and $d(\lambda'') = d(\lambda)-n$. We write
$\lambda(n)$ for $\lambda(n,n)=s(\lambda(0,n))$. We denote by $\Omega_k$ the $k$-graph
with vertices $\Omega^0_k:=\NN^k$, paths  $\Omega^m_k:=\{(n,n+m), n\in \NN^k\}$ for $m\in
\NN^k$, $r((n,n+m))=n$ and $s((n,n+m))=n+m$.

Let $\Lambda$ be a row-finite $k$-graph with no sources. A \emph{Cuntz--Krieger
$\Lambda$-family} in a $C^*$-algebra $B$ is a function $s : \lambda \mapsto s_\lambda$
from $\Lambda$ to $B$ such that
\begin{itemize}
\item[(CK1)] $\{s_v : v\in \Lambda^0\}$ is a collection of mutually orthogonal
    projections;
\item[(CK2)] $s_\mu s_\nu = s_{\mu\nu}$ whenever $s(\mu) = r(\nu)$;
\item[(CK3)] $s^*_\lambda s_\lambda = s_{s(\lambda)}$ for all $\lambda\in \Lambda$;
    and
\item[(CK4)] $s_v = \sum_{\lambda \in v\Lambda^n} s_\lambda s_\lambda^*$ for all
    $v\in \Lambda^0$ and $n\in \NN^k$.
\end{itemize}
The \emph{$k$-graph $C^*$-algebra $C^*(\Lambda)$} is the universal $C^*$-algebra
generated by a Cuntz--Krieger $\Lambda$-family. For more details see \cite{KumPas}.

\section{Topological dimension zero}\label{sec:topdim}
In this section we study when $k$-graph $C^*$-algebras have topological dimension zero.
This requires some machinery, which we now introduce. Let $\Lambda$ be a row-finite
$k$-graph with no sources. A subset $H$ of $\Lambda^0$ is \emph{hereditary} if
$s(H\Lambda)\subseteq H$. We say $H\subseteq \Lambda^0$ is \emph{saturated} if for all
$v\in \Lambda^0$
$$\{s(\lambda)  : \lambda\in v\Lambda^{e_i}\}\subseteq H \text{ for some } i\in \{1,\dots,k\} \qquad\Longrightarrow\qquad v\in H.$$
(Since $\Lambda$ has no sources, we do not need to worry about whether $r^{-1}(v)$ is
empty---cf.~\cite{RaeSimYee}.) For a hereditary $H\subseteq \Lambda^0$ we write $I_H$ for
the ideal in $C^*(\Lambda)$ generated by $\{s_v : v\in H\}$. If $H$ is hereditary, then
$H\Lambda$ is a row-finite $k$-graph with no sources, and if $H$ is also saturated, then
$\Lambda\setminus \Lambda H$ is also a row-finite $k$-graph with no sources (see
\cite{KumPas}).

Let $\Lambda$ be a row-finite $k$-graph with no sources. The set
\[
\Lambda^\infty:=\{x : \Omega_k\to \Lambda \mid x \textrm{ is a degree-preserving functor}\}
\]
is called the \emph{infinite-path space} of $\Lambda$. For $v\in \Lambda^0$, we write
$v\Lambda^\infty:=\{x\in \Lambda^\infty: x(0)=v\}$, and for $x\in\Lambda^\infty$, we
write $r(x):=x(0)\in \Lambda^0$. For $p\in \NN^k$, the shift map $\sigma^p :
\Lambda^\infty\to \Lambda^\infty$ defined by $(\sigma^p x)(m,n)=x(m+p,n+p)$ for
$(m,n)\in\Omega_k$ is a local homeomorphism. For $\lambda\in \Lambda$ and $x\in
\Lambda^\infty$ with $s(\lambda)=r(x)$ we write $\lambda x$ for the unique element $y\in
\Lambda^\infty$ such that $\lambda=y(0,d(\lambda))$ and $x=\sigma^{d(\lambda)}y$, see
\cite{KumPas}.

\begin{defn}
Let $\Lambda$ be a row-finite $k$-graph with no sources. We say that $\Lambda$ is
\emph{aperiodic} (or satisfies the \emph{aperiodicity condition}) if for every vertex
$v\in \Lambda^0$ there exist an infinite path $x\in v\Lambda^\infty$ such that
$\sigma^m(x)\neq \sigma^n(x)$ for all $m\neq n\in \NN^k$, see \cite{KumPas, RaeSimYee,
RobSim}. If $\Lambda\setminus \Lambda H$ is aperiodic for every hereditary saturated
$H\subsetneq \Lambda^0$, we say $\Lambda$ is \emph{strongly aperiodic}, see
\cite{KanPas}.
\end{defn}

We can now state our main result about topological dimension of $k$-graph $C^*$-algebras.
Let us emphasise that half of the hard work has already been done for us: The implication
(ii)$\Rightarrow$(i) was established in \cite{KanPas} by Kang and the first named author.
As we shall see the reverse implication (i)$\Rightarrow$(ii) also requires some
non-trivial work.

\begin{thm}\label{thm2.3}
Let $\Lambda$ be a row-finite $k$-graph with no sources. Then the following are
equivalent:
\begin{itemize}
\item[(i)] The $C^*$-algebra $C^*(\Lambda)$ has topological dimension zero.
\item[(ii)] The $k$-graph $\Lambda$ is strongly aperiodic.
\end{itemize}
\end{thm}

Before we prove Theorem~\ref{thm2.3} we present six lemmata. The key ideas, namely the
use of $P$-graphs and induced algebras, are introduced in
Lemmas~\ref{lem2.7}~and~\ref{lem2.8}. We briefly recall the notation involved.

Following \cite{KanPas} we recall the definition of a maximal tail. Let $\Lambda$ be a
row-finite $k$-graph with no sources. A nonempty subset $T$ of $\Lambda^0$ is called a
\emph{maximal tail} if
\begin{itemize}
\item[(a)] for all $v_1,v_2\in T$ there exists $w\in T$ such that $v_1\Lambda w\neq
    \emptyset$ and $v_2\Lambda w\neq \emptyset$,
\item[(b)] for every $v\in T$ and $1\leq i\leq k$ there exist $e\in v\Lambda^{e_i}$
    such that $s(e)\in T$, and
\item[(c)] for all $w\in T$ and $v\in \Lambda^0$ with $v\Lambda w\neq \emptyset$ we
    have $v\in T$.
\end{itemize}

Following \cite{CarKanShoSim} we recall the definition of $\Per(\Lambda)$ and $H_{\Per}$.
Let $\Lambda$ be a row-finite $k$-graph with no sources such that $\Lambda^0$ is a
maximal tail. Define a relation $\sim$ on $\Lambda$ by $\mu\sim\nu$ if and only if
$s(\mu)=s(\nu)$ and $\mu x=\nu x$ for all $x\in s(\mu)\Lambda^\infty$. Then
\cite[Lemma~4.2(1)]{CarKanShoSim} says that the set $\Per(\Lambda):=\{d(\mu)-d(\nu) :
\mu,\nu\in \Lambda \textrm{ and } \mu\sim\nu\}$ is a subgroup of $\ZZ^k$, called the
\emph{periodicity group} of $\Lambda$. Moreover, it follows from
\cite[Lemma~4.6]{CarKanShoSim}, that
\begin{align*}
\Per(\Lambda)=\{l\in \ZZ^k: \textrm{ there exists } w\in \Lambda^0 \textrm{ and }n,m\in \NN^k &\textrm{ such that }l=n-m \text{ and } \\
    &\sigma^{n}(y)=\sigma^{m}(y) \textrm{ for all } y\in w\Lambda^\infty\} .
\end{align*}
Let
\begin{align*}
H_{\Per}:=\{v\in \Lambda^0 : \textrm{ for all } \lambda\in v\Lambda \textrm{ and } m\in \NN^k& \textrm{ such that } d(\lambda)-m \in \Per(\Lambda),\\
    &\textrm{ there exists } \mu\in v\Lambda^m \textrm{ such that } \lambda\sim\mu\}.
\end{align*}
Lemma~4.2 of \cite{CarKanShoSim} shows that $H_{\Per}$ is a non-empty hereditary set (not
always saturated).

We will also need to use the $P$-graphs which were introduced in \cite{CarKanShoSim}. The
definitions of $P$-graphs, where $P$ is a finitely generated cancellative abelian monoid,
and the associated $C^*$-algebras are obtained by replacing $\NN^k$ with $P$ in the
definitions for $k$-graphs given above, see \cite[Section~2]{CarKanShoSim} for details.
If $\Gamma$ is a $P$-graph we will continue to call a family satisfying (CK1)--(CK4) for
$\Gamma$ with $P$ replacing $\NN^k$ a Cuntz--Krieger $\Gamma$-family. If
\[
P:=\{m+\Per(\Lambda) : m\in \NN^k\}\subseteq \ZZ^k/\Per(\Lambda),
\]
then $(H_{\Per}\Lambda)/\negthickspace\sim$ is a $P$-graph with operations inherited from
$\Lambda$ and degree map given by $\tilde{d}([\lambda])=d(\lambda)+\Per(\Lambda)$.

\begin{lemma}\label{lem2.7}
Let $\Lambda$ be a row-finite $k$-graph with no sources such that $\Lambda^0$ is a
maximal tail. Define $H:=H_{\Per}$ and let $\Gamma:=(H\Lambda)/\negthickspace\sim$ be the
$P$-graph described above. There exist an action $\alpha$ of the annihilator
$\Per(\Lambda)^\bot:=\{h\in \Hom(\ZZ^k,\TT) : h|_{\Per(\Lambda)}=1\}$ on $C^*(\Gamma)$
such that
$$\alpha_h(s_{[\lambda]})=h(d(\lambda))^{-1}s_{[\lambda]}, \quad\text{ for all }\quad h\in \Per(\Lambda)^\bot \text{ and } \lambda\in H\Lambda.$$
\end{lemma}
\begin{proof}
Fix $h\in \Per(\Lambda)^\bot$ and $\lambda,\mu\in H\Lambda$ such that $[\lambda]=[\mu]$.
We prove that $h(d(\lambda))=h(d(\mu))$. Since $\lambda\sim \mu$ we have
$s(\lambda)=s(\mu)$ and $\lambda x = \mu x$ for all $x\in s(\lambda)(H\Lambda)^{\infty}$.
Set $w=s(\lambda)\in H$ (since $H$ is hereditary). For any $x\in w\Lambda^{\infty}$ we
have $x\in s(\lambda)(H\Lambda)^{\infty}$ so
$\sigma^{d(\lambda)}(x)=\sigma^{d(\lambda)+d(\mu)}(\mu
x)=\sigma^{d(\lambda)+d(\mu)}(\lambda x)=\sigma^{d(\mu)}(x)$. In particular
$d(\lambda)-d(\mu)\in \Per(\Lambda)$. Since $h\in \Per(\Lambda)^\bot$ we see that
$h(d(\lambda)-d(\mu))=1$, so $h(d(\lambda))=h(d(\mu))$.

Fix $h\in \Per(\Lambda)^\bot$. By the preceding paragraph, for $[\lambda]\in \Gamma$, we
can define
$$t_{[\lambda]}:=h(d(\lambda))^{-1}s_{[\lambda]}.$$
Let $P:=\{m+\Per(\Lambda) : m\in \NN^k\}\subseteq \ZZ^k/\Per(\Lambda)$. Since $\Gamma$ is
a $P$-graph and $C^*(\Gamma)$ is the universal algebra generated by a Cuntz--Krieger
$\Gamma$-family $\{s_{[\lambda]} : [\lambda]\in \Gamma=(H\Lambda)/\negthickspace\sim\}$,
the set $\{t_{[\lambda]} : [\lambda]\in \Gamma\}$ is also a Cuntz--Krieger
$\Gamma$-family. So the universal property yields a homomorphism $\alpha_h :
C^*(\Gamma)\to C^*(\Gamma)$ such that
$\alpha_h(s_{[\lambda]})=t_{[\lambda]}=h(d(\lambda))^{-1}s_{[\lambda]}$.

For each $g,h\in \Per(\Lambda)^\bot$, $\alpha_{hg}=\alpha_h\circ \alpha_g$ and
$\alpha_{hh^{-1}}=\id_{C^*(\Gamma)}$ on the generators of $C^*(\Gamma)$, and hence on all
of $C^*(\Gamma)$. It follows that $h\mapsto \alpha_h$ is a group homomorphism from
$\Per(\Lambda)^\bot$ into the automorphisms of $C^*(\Gamma)$.

Finally we show $h\mapsto \alpha_h(a)$ is continuous for each $a\in C^*(\Gamma)$. Let
$a=s_{[\lambda]}$ for some $\lambda\in H\Lambda$ and suppose $h_i\to h$ in
$\Per(\Lambda)^\bot$. By definition of the topology on $\Hom(\ZZ^k,\TT)$ we have
$h_i(d(\lambda))\to h(d(\lambda))$ in $\TT$, and so $\alpha_{h_i}(a)\to \alpha_{h}(a)$ in
$C^*(\Gamma)$. Since addition is continuous, we see that $\alpha_{h_i}(a)\to
\alpha_{h}(a)$ for $a\in\espan\{t_\eta t_\nu^* : \eta,\nu\in\Gamma\}$; an
$\varepsilon/3$-argument then gives continuity for each $a\in C^*(\Gamma)$.
\end{proof}

Following \cite[p.~100]{Wil} let $X$ be a right $G$-space and $(A,G,\alpha)$ a dynamical
system. If $f : X\to A$ is a continuous function such that
\begin{align}
\label{eqn2.1}
f(x\cdot s) = \alpha_{s}^{-1}(f(x)), \quad\text{ for all } x\in X  \textrm{ and } s\in G	
\end{align}
then $x\mapsto \|f(x)\|$ is continuous on the orbit space $X/G:=\{x\cdot G : x\in X\}$.
The \emph{induced algebra} is
\[
\Ind_{G}^{X}(A, \alpha)
    = \{f\in C_b(X,A): f \textrm{ satisfies~\eqref{eqn2.1} and }
            x\cdot G \mapsto \|f(x)\| \textrm{ is in }C_0(X/G)\}.
\]
For $n\in \NN^k$ and $z\in \TT^k$, we write $z^n$ for $\prod_{i=1}^k z_i^{n_i}$.

In what follows we identify $\Per(\Lambda)^\bot$ with $\{z\in \TT^k : z^n=1 \textrm{ for
all } n\in \Per(\Lambda)\}\subseteq \TT^k$. Since $\Per(\Lambda)^\bot$ acts on $\TT^k$ by
right multiplication $z\cdot w:= zw$ \cite[Example~3.34]{Wil}, $\TT^k$ is a right
$\Per(\Lambda)^\bot$-space. So, using the action $\alpha : \Per(\Lambda)^\bot \to
\operatorname{Aut}(C^*(\Gamma))$ of Lemma~\ref{lem2.7},  we can form the induced algebra
$\Ind(\alpha):=\Ind_{\Per(\Lambda)^\bot}^{\TT^k}(C^*(\Gamma),\alpha)$.

Recall that if $\Lambda$ is a row-finite $k$-graph with no sources, then the universal
property of $C^*(\Lambda)$ ensures that it carries a canonical action $\gamma$, called
the \emph{gauge action}, of $\TT^k$ satisfying $\gamma_z(s_\lambda) = z^{d(\lambda)}
s_\lambda$ (see \cite{KumPas}).

\begin{lemma}\label{lem2.8}
Let $\Lambda$, $H$, $\alpha$ and $\Gamma$ be as in Lemma~\ref{lem2.7}. Then there exists
an isomorphism $\pi : C^*(H\Lambda)\to \Ind(\alpha)$ such that
$\pi(s_\lambda)(z)=z^{d(\lambda)}s_{[\lambda]}$ for each $\lambda\in H\Lambda$ and $z\in
\TT^k$. We have
\[
\Prim(C^*(H\Lambda))\cong \Big(\TT^k \times \Prim(C^*(\Gamma))\Big)/\Per(\Lambda)^\bot,
\]
where the action of $\Per(\Lambda)^\bot$ on $\TT^k \times \Prim(C^*(\Gamma))$ is given by
$(z,I)\cdot w=(zw,\alpha_{w^{-1}}(I))$.
\end{lemma}

\begin{proof}
For each $\lambda\in H\Lambda$ define $t_{\lambda} : \TT^k\to C^*(\Gamma)$ by
$t_{\lambda}(z)= z^{d(\lambda)}s_{[\lambda]}$. We show that each $t_{\lambda}$ belongs to
the $C^*$-algebra $\Ind(\alpha)$ described above.

For fixed $\lambda\in H\Lambda$, $z\mapsto z^{d(\lambda)}$ is continuous from $\TT^k$ to
$\TT$, so $t_\lambda$ belongs to $C(\TT^k,C^*(\Gamma))$. Fix $z\in \TT^k$ and $w\in
\Per(\Lambda)^\bot$. We have
\begin{align*}
\alpha_{w}^{-1}(t_\lambda(z))&= \alpha_{w}^{-1}(z^{d(\lambda)}s_{[\lambda]})
= z^{d(\lambda)}\alpha_{ w^{-1} }(s_{[\lambda]})\\
&= z^{d(\lambda)}w^{d(\lambda)}s_{[\lambda]}=(zw)^{d(\lambda)}s_{[\lambda]}=t_\lambda(zw)=t_\lambda(z\cdot w).
\end{align*}
So $t_\lambda$ satisfies $(\ref{eqn2.1})$. The function on $\TT^k/\Per(\Lambda)^\bot$
given by $z\cdot \Per(\Lambda)^\bot \mapsto \|t_\lambda(z)\|$ is constant, hence
continuous. Since $\TT^k/\Per(\Lambda)^\bot$ is compact and Hausdorff (see
\cite[p.~101]{Wil}), the map $z\cdot \Per(\Lambda)^\bot \mapsto \|t_\lambda(z)\|$ belongs
to $C_0(\TT^k/\Per(\Lambda)^\bot)$, so $t_\lambda\in \Ind(\alpha)$.

Since $H\Lambda$ is a $k$-graph and $C^*(H\Lambda)$ is the universal algebra generated by
the Cuntz--Krieger family $\{s_{\lambda} : \lambda\in H\Lambda\}$ it easily follows that
$\{t_{\lambda} : \lambda\in H\Lambda\}$ satisfies (CK1)--(CK3). To verify (CK4) fix $v\in
H$ and $n\in \NN^k$, let $p:=n+\Per(\Lambda)$, and define $\varphi : v\Lambda^n\to
v\Gamma^p$, by $\varphi(\lambda)=[\lambda]$. It suffices to show that $\varphi$ is a
bijection. To check that $\varphi$ is surjective (this is non-trivial because
$[\lambda]\in v\Gamma^p$ $\not\Rightarrow$ $d(\lambda)=n$), choose $[\lambda]\in
v\Gamma^p$. Since $\tilde{d}([\lambda])=d(\lambda)+\Per(\Lambda)$ there exists $\mu\sim
\lambda$ such that $d(\mu)=n$ and $r(\mu)=r(\lambda)$ (see
\cite[Theorem~4.2]{CarKanShoSim}). We obtain $\mu\in v\Lambda^n$ and
$\varphi(\mu)=[\lambda]$. To check that $\varphi$ is injective, suppose that
$\varphi(\lambda)=\varphi(\mu)$ for some $\lambda,\mu\in v\Lambda^n$. Since $\lambda\sim
\mu$ and $d(\lambda)=d(\mu)$ it follows from \cite[Theorem~4.2]{CarKanShoSim} that
$\lambda=\mu$. So $\varphi$ is a bijection as claimed.

By the universal property of $C^*(H\Lambda)$ there exists a homomorphism $\pi :
C^*(H\Lambda)\to \Ind(\alpha)$ such that
\begin{align}\label{eqn3.2}
\pi(s_\lambda)=t_\lambda= \left(z\mapsto z^{d(\lambda)}s_{[\lambda]}\right), \quad\text{ for each }\lambda\in H\Lambda.
\end{align}
We now show $\pi$ is injective. Similarly to \cite[Lemma~3.48]{Wil}, we have a dynamical
system $(\Ind(\alpha),\TT^k,\lt)$ where $\lt_{z}(f)(x)=f(xz)$ for all $z\in \TT^k$ and
$f\in \Ind(\alpha)$. Writing $\gamma$ for the gauge action on $C^*(H\Lambda)$, for
$x,z\in \TT^k$ and $\lambda\in H\Lambda$, we have
\begin{align*}
\lt_z\circ\, \pi(s_\lambda)(x)&=\lt_z(t_{\lambda})(x)=t_{\lambda}(xz)=(xz)^{d(\lambda)}s_{[\lambda]}=z^{d(\lambda)} x^{d(\lambda)}s_{[\lambda]}\\
&=z^{d(\lambda)}t_\lambda(x)=z^{d(\lambda)}\pi(s_\lambda)(x)=\pi(z^{d(\lambda)} s_\lambda)(x)=\pi\circ\gamma_{z}(s_\lambda)(x).
\end{align*}
So the homomorphisms $\lt_z\circ\, \pi$ and $\pi\circ\gamma_z$ agree on generators, and
hence are equal. Since each $t_\lambda$ is non-zero it follows from the gauge-invariant
uniqueness theorem \cite[Theorem~3.4]{KumPas} that $\pi$ is injective.

We now show that $\pi$ is surjective. By \cite[Proposition~3.49]{Wil}, $\Ind(\alpha)$ can
be regarded as a $C_0(\TT^k/\Per(\Lambda)^\bot)$-algebra in the sense of \cite{Wil},
where
\begin{align}
\label{b.ref1}
\varphi \cdot f(z) = \varphi(z\cdot \Per(\Lambda)^\bot)f(z)
\end{align}
for all $\varphi\in C_0(\TT^k/\Per(\Lambda)^\bot)$, $f\in\Ind(\alpha)$ and $z\in \TT^k$.
In particular it follows from \cite[Proposition~C.24 and Theorem~C.26]{Wil} that $\pi$ is
surjective provided the following two properties hold:
\begin{itemize}
\item[(i)] If $f\in \im(\pi)$ and $\varphi\in C_0(\TT^k/\Per(\Lambda)^\bot)$, then
    $\varphi \cdot f\in \im(\pi)$.
\item[(ii)] For each $z\in \TT^k$, $\{f(z) : f\in \im(\pi)\}$ is dense in
    $C^*(\Gamma)$.
\end{itemize}
To prove (i), observe that for $n \in \Per(\Lambda)$, the map $z \mapsto z^n$ vanishes on
$\Per(\Lambda)^\bot$, and so descends to an element $\varepsilon_n$ of
$C_0(\TT^k/\Per(\Lambda)^\bot)$ satisfying
\begin{align}
\label{b.ref2}
\varepsilon_n(z\cdot \Per(\Lambda)^\bot) = z^n.
\end{align}
The set $\{\varepsilon_n: n \in \Per(\Lambda)\}$ generates
$C_0(\TT^k/\Per(\Lambda)^\bot)$ as a $C^*$-algebra. So (i) follows once we prove that
$\im(\pi)$ is invariant for the action on every $\varepsilon_n$. For $f\in\Ind(\alpha)$ and $z
\in \TT^k$, equations \eqref{b.ref1}~and~\eqref{b.ref2} give
$$(\varepsilon_n \cdot f)(z) = \varepsilon_n(z\cdot \Per(\Lambda)^\bot)f(z)=z^n f(z).$$
Fix $n \in \Per(\Lambda)$ and $f \in \im(\pi)$. We show that the function $z \mapsto z^n
f(z)$ belongs to $\im(\pi)$. Write $n=p-q$ with $p,q\in \NN^k$ and fix $\lambda,\mu\in
H\Lambda$. By \cite[Theorem~4.2(3)]{CarKanShoSim} there exist a bijection $\theta :
s(\lambda)\Lambda^p\to s(\lambda)\Lambda^q$ such that $\nu\sim \theta(\nu)$ for all $\nu
\in s(\lambda)\Lambda^p$. Since $\{s_{[\lambda]} : [\lambda]\in \Gamma\}$ satisfies (CK4)
and the map $\nu\mapsto[\nu]$ is a bijection from $s(\lambda)\Lambda^p$ to
$s(\lambda)\Gamma^{p+\Per(\Lambda)}$ we have,
\begin{align*}
t_{s(\lambda)}(z)
    &= z^{d(s(\lambda))}s_{[s(\lambda)]}
     =s_{[s(\lambda)]}
     =\sum_{[\nu]\in s(\lambda)\Gamma^{p+\Per(\Lambda)}} s_{[\nu]}s_{[\nu]}^*\\
    &=\sum_{\nu\in s(\lambda)\Lambda^{p}} s_{[\nu]}s_{[\nu]}^*
     =\sum_{\nu\in s(\lambda)\Lambda^p} s_{[\nu]}s_{[\theta(\nu)]}^*\\
    &=\sum_{\nu\in s(\lambda)\Lambda^p} z^{-d(\nu)}t_{\nu}(z)(z^{-d(\theta(\nu))}t_{\theta(\nu)}(z))^*\\
    &=z^{-(p-q)}\sum_{\nu\in s(\lambda)\Lambda^p}t_{\nu}t_{\theta(\nu)}^*(z).
\end{align*}
Hence $\left(z\mapsto z^n t_\lambda t_\mu^*(z)\right)=\left(z\mapsto \big(\sum_{\nu\in
s(\lambda)\Lambda^p}t_{\lambda}t_{\nu}t_{\theta(\nu)}^*t_\mu^*\big)(z)\right)\in\im(\pi)$.
Finally using that $f\in \ospan\{t_\lambda t_\mu^* : \lambda,\mu\in H\Lambda\}$ we obtain
$(z \mapsto z^n f(z))\in\im(\pi)$.

To prove (ii), fix $z\in \TT^k$. For each $\lambda\in H\Lambda$ define
$f_\lambda:=z^{-d(\lambda)}t_{\lambda}\in \im(\pi)$. Using \eqref{eqn3.2} we get
$s_{[\lambda]}=f_\lambda(z)$, so $s_{[\lambda]}\in \{f(z) : f\in \im(\pi)\}$. Since
$\im(\pi)$ is a $C^*$-algebra and since $\ospan\{f_\lambda f_\mu^*(z) : \lambda,\mu\in
H\Lambda\}=C^*(\Gamma)$, the set $\{f(z) : f\in \im(\pi)\}$ is dense in $C^*(\Gamma)$. We
conclude that $C^*(H\Lambda)\cong \Ind(\alpha)$.

By \cite[p.~100]{Wil}, the primitive ideal space of the induced algebra $\Ind(\alpha)$ is
homeomorphic with the orbit space $(\TT^k \times \Prim(C^*(\Gamma)))/\Per(\Lambda)^\bot$,
where the right action of $\Per(\Lambda)^\bot$ is given by $(z,I)\cdot
w:=(zw,\alpha_{w^{-1}}(I))$ \cite[Lemma~2.8]{Wil}. In particular, since
$C^*(H\Lambda)\cong \Ind(\alpha)$, we obtain $\Prim(C^*(H\Lambda))\cong
\Prim(\Ind(\alpha))\cong(\TT^k \times \Prim(C^*(\Gamma)))/\Per(\Lambda)^\bot$.
\end{proof}

\begin{lemma}\label{lem2.10}
Let $\Lambda$ and $\Gamma$ be as in Lemma~\ref{lem2.7}. The formula $$\varphi((z,I)\cdot
\Per(\Lambda)^\bot)=z\cdot\Per(\Lambda)^\bot,$$ determines a continuous open surjection
$\varphi : \big(\TT^k \times \Prim(C^*(\Gamma))\big)/\Per(\Lambda)^\bot \to
\TT^k/\Per(\Lambda)^\bot$.
\end{lemma}

\begin{proof}
For completeness let us mention that $\Prim(C^*(\Gamma))$ is non-empty (for example, the
zero ideal in $C^*(\Gamma)$ is primitive) so $\varphi$ makes sense. If
$(x_1,I_1)\cdot\Per(\Lambda)^\bot=(x_2,I_2)\cdot\Per(\Lambda)^\bot$ for $(x_i,I_i)\in
\TT^k \times \Prim(C^*(\Gamma))$, then $x_1w=x_2$ for some $w\in \Per(\Lambda)^\bot$. So
the map $\varphi$ is well defined.

Let $\psi : \TT^k \times \Prim(C^*(\Gamma))\to \TT^k$ be the projection onto the first
coordinate and write $p$ for both the quotient map $\TT^k \to \TT^k/\Per(\Lambda)^\bot$
and the orbit map $\TT^k \times \Prim(C^*(\Gamma))\to (\TT^k \times
\Prim(C^*(\Gamma)))/\Per(\Lambda)^\bot$ (cf.~\cite[Definition 3.21]{Wil}). Then the
diagram
\[
\xymatrix{{(\TT^k \times \Prim(C^*(\Gamma)))/\Per(\Lambda)^\bot} \ar[rrr]^-\varphi &&& \TT^k/\Per(\Lambda)^\bot \\
\TT^k \times \Prim(C^*(\Gamma)) \ar[rrr]^-\psi \ar[u]^p  &&& \TT^k \ar[u]^p},
\]
commutes. The topologies on the orbit spaces
\[
(\TT^k \times \Prim(C^*(\Gamma)))/\Per(\Lambda)^\bot \quad\text{ and }\quad
\TT^k/\Per(\Lambda)^\bot
\]
are the weakest topologies making the maps $p$ continuous. So both maps $p$ are
continuous and open (see \cite[Lemma~3.25]{Wil}). Now $\varphi$ is continuous and open
because $\psi$ is. The map $\varphi$ is surjective because $\Prim(C^*(\Gamma))$ is
non-empty.
\end{proof}

We will need an elementary lemma from topology. This is standard, but we give a complete
statement and proof for convenience.

\begin{lemma}\label{lem2.4}
Let $X,Y$ be topological spaces and let $\varphi : X \to Y$ be a continuous open
surjective map. If $X$ admits a basis for its topology consisting of compact and open
sets, then so does $Y$.
\end{lemma}
\begin{proof}
Let $\BB$ be a basis for the topology on $X$ consisting of compact and open sets. Set
$\FF:=\{\varphi(B):B\in\BB\}$. Since $\varphi$ is continuous and open each element of
$\FF$ is compact and open. Since $\varphi$ is surjective
$Y=\varphi(\bigcup_{B\in\BB}B)=\bigcup_{F\in\FF}F$, so $\FF$ is an open cover of $Y$.
Suppose $y\in F_1\cap F_2$ for $F_1,F_2\in \FF$. Take $x\in X$ such that $\varphi(x)=y$.
Since $x\in \varphi^{-1}(y)\subseteq \varphi^{-1}(F_i)$ and $F_i$ is open, there exists
$B\in \BB$ such that $x\in B\subseteq \varphi^{-1}(F_1)\cap \varphi^{-1}(F_2)$. With
$F:=\varphi(B)\in \FF$ we get $y=\varphi(x)\in \varphi(B)=F$ and $F=\varphi(B)\subseteq
\varphi(\varphi^{-1}(F_i))=F_i$. So $\FF$ is a basis for a topology on $Y$.

To see that this topology coincides with the given topology on $Y$, observe that each set
in $\FF$ is open, and for any open set $V\subseteq Y$, the preimage $\varphi^{-1}(V)$ is
open, hence a union of elements $B_i \in \BB$. Now each $\varphi(B_i)\in \FF$ and
$V=\varphi(\varphi^{-1}(V))=\varphi(\bigcup_iB_i))=\bigcup_i\varphi(B_i)$.
\end{proof}

We write $\mathcal{M}(A)$ for the multiplier algebra of a $C^*$-algebra $A$.

\begin{lemma}\label{lem2.6}
Let $\Lambda$ be a row-finite $k$-graph with no sources. Suppose that $C^*(\Lambda)$ has
topological dimension zero. Suppose that $H \subseteq \Lambda^0$ is hereditary. Let $I$
be the ideal of $C^*(\Lambda)$ generated by $\{p_v : v \in H\}$, and let $P_H := \sum_{v
\in H} p_v \in \mathcal{M}(I)$. There is an isomorphism $C^*(H\Lambda) \cong P_H I P_H$
carrying the generator $s_\lambda$ of $C^*(H \Lambda)$ to the corresponding generator
$s_\lambda$ of $C^*(\Lambda)$ for every $\lambda \in H\Lambda$, and $P_H I P_H$ is a full
corner of $I$. Furthermore, $C^*(H\Lambda)$ has topological dimension zero. If $H$ is
also saturated then $C^*(\Lambda\setminus \Lambda H)$ is isomorphic to $C^*(\Lambda)/I$
and has topological dimension zero.
\end{lemma}
\begin{proof}	
Recall from \cite[Proposition~3.2.1]{Dix} and \cite[Proposition~2.8]{BroPed2} that
topological dimension zero passes to ideals and quotients. It also passes to full corners
(since Morita equivalent $C^*$-algebras have homeomorphic primitive ideal spaces, see
\cite[p.~156]{Con}).

Let $H$ be any hereditary subset of $\Lambda^0$. By \cite[Theorem~5.2]{RaeSimYee},
$C^*(H\Lambda)$ is isomorphic to $P_H I P_H$. This is a full corner because $I$ is
generated by $P_H$. Since topological dimension zero passes to ideals and full corners,
we deduce that $C^*(H\Lambda)$ has topological dimension zero. If $H$ is saturated as
well, then $C^*(\Lambda\setminus \Lambda H)$ is isomorphic to $C^*(\Lambda)/I$, and hence
it, too, has topological dimension zero.
\end{proof}

\begin{lemma}\label{lem2.9}
Let $\Lambda$ be a row-finite $k$-graph with no sources. Suppose that $\Lambda$ is not
strongly aperiodic. Then there exists a (possibly empty) hereditary and saturated set
$H\subseteq {\Lambda}^{0}$ such that the vertex set of $\Lambda\setminus \Lambda H$ is a
maximal tail and $\Per(\Lambda\setminus \Lambda H)\neq \{0\}$.
\end{lemma}
\begin{proof}
Choose a (possibly empty) hereditary saturated set $H'\subseteq{\Lambda}^0$ such that
$\Lambda\setminus \Lambda H'$ is not aperiodic. By \cite[Lemma~3.2(iii)]{RobSim} there
exist a vertex $v\in \Lambda\setminus \Lambda H'$ and $n\neq m\in \NN^k$ such that
$\sigma^m(y)=\sigma^n(y)$ for every infinite path $y\in v(\Lambda\setminus \Lambda
H')^{\infty}$. Fix $x\in v(\Lambda\setminus \Lambda H')^{\infty}$, define the shift-tail
equivalence class
\begin{align}
\label{eqn2.4}
[x]:=\{\lambda\sigma^p(x) : p\in \NN^k, \lambda\in (\Lambda\setminus \Lambda H')x(p)\},
\end{align}
and let $T:=r([x])\subseteq \Lambda^0$ (possibly all of $\Lambda^0$) and
$H:=\Lambda^0\setminus T$.

We prove that $T\subseteq \Lambda^0$ is a maximal tail. First we consider property (a).
If $v,w\in r([x])$, say $v=r(\alpha\sigma^m(x))$, $w=r(\beta\sigma^n(x))$, then
$u=x(m\vee n) \in T$ satisfies $\alpha{x}(m,m\vee n)\in v\Lambda{u}$, and
$\beta{x}(n,m\vee n)\in w\Lambda{u}$. For property (b), fix $v\in r([x])$, say
$v=r(\alpha\sigma^m(x))$, and fix $i\leq k$. By replacing $\alpha$ with
$\alpha{x}(m,m+e_i)$ and $m$ with $m+e_i$, we may assume that $d(\alpha)\geq e_i$, say
$\alpha=f\alpha'$ with $f\in \Lambda^{e_i}$. Then $f\in v\Lambda^{e_i}$, and
$s(f)=r(\alpha'\sigma^m(x))\in r([x])$. Finally for property (c), fix $w\in r([x])$ and
$v\in \Lambda^0$ with $v\Lambda w\neq \emptyset$, say $\alpha\in v\Lambda{w}$. Write
$w=r(\beta\sigma^m(x))$. Then $v=r((\alpha\beta)\sigma^m(x))\in r([x])$.

Now $H\subseteq \Lambda^0$ is hereditary and saturated because it is the complement of a
maximal tail (see the proof of \cite[Theorem~3.12]{KanPas}). Since $T(\Lambda\setminus
\Lambda H)T=\Lambda\setminus \Lambda H$, the set $T=(\Lambda\setminus \Lambda H)^0$ is a
maximal tail in the $k$-graph $\Lambda\setminus \Lambda H$.

Since $\Per(\Lambda\setminus \Lambda H)=\{a-b: \textrm{ there exists } w\in
(\Lambda\setminus \Lambda H)^0 \textrm{ such that }\sigma^{a}(y)=\sigma^{b}(y)$ for all
$y\in w(\Lambda\setminus \Lambda H)^\infty\}$ and since $(\Lambda\setminus \Lambda
H)^0=r([x])\subseteq \{r(\lambda) : \lambda\in\Lambda\setminus \Lambda H'\}=
(\Lambda\setminus \Lambda H')^0$ we conclude that $n-m\in\Per(\Lambda\setminus \Lambda
H)\neq \{0\}$.
\end{proof}

\begin{proof}[Proof of Theorem~\ref{thm2.3}] As mentioned before,
(ii)$\Rightarrow$(i) follows from \cite[Theorem~4.2]{KanPas}. For (i)$\Rightarrow$(ii),
let $\Lambda'$ be a row-finite $k$-graph with no sources such that $C^*(\Lambda')$ has
topological dimension zero. We suppose that $\Lambda'$ is not strongly aperiodic, and
derive a contradiction. By Lemma~\ref{lem2.9} there is a hereditary and saturated set
$H'\subseteq {(\Lambda')}^{0}$ such that the vertex set $\Lambda^0$ of
$\Lambda:=\Lambda'\setminus \Lambda' H'$ is a maximal tail and such that $\Per(\Lambda)$
is nontrivial. Define $H:=H_{\Per}$ and $\Gamma:=(H\Lambda)/\negthickspace\sim$. By
Lemma~\ref{lem2.8}, $\Prim(C^*(H\Lambda))\cong \big(\TT^k \times
\Prim(C^*(\Gamma))\big)/\Per(\Lambda)^\bot$. By Lemma~\ref{lem2.10}, there is a
continuous open surjective map
$$\varphi : \big(\TT^k \times \Prim(C^*(\Gamma))\big)/\Per(\Lambda)^\bot \to \TT^k/\Per(\Lambda)^\bot.$$
Lemma~\ref{lem2.6} implies that $C^*(H\Lambda)$ has topological dimension zero, and it
follows that $\big(\TT^k \times \Prim(C^*(\Gamma))\big)/\Per(\Lambda)^\bot$ has a basis
of compact open sets. Using Lemma~\ref{lem2.4} we conclude that
$\TT^k/\Per(\Lambda)^\bot$ has a basis of compact open sets. But
$\TT^k/\Per(\Lambda)^\bot$ is the Pontryagin dual of a free abelian group, and hence
homeomorphic to $\TT^l$ where $l$ is the rank of $\Per(\Lambda)$. So
$\TT^k/\Per(\Lambda)^\bot$ is connected and not a singleton, a contradiction.
\end{proof}

Recall that a topological groupoid $\Gg$ is \emph{topologically principal} if the set
$$\{u\in \Gg^{(0)} : \Gg^{u}_{u}=\{u\}\}$$ of units with trivial isotropy is dense in
$\Gg^{(0)}$ \cite[Definition 1.2]{Ren2}, and \emph{essentially principal} if for every
closed invariant $X\subseteq \Gg^{(0)}$, the set $\{u\in X : \Gg^{u}_{u}=\{u\}\}$ is
dense in $X$ \cite[Definition~4.3]{Ren}. For each $k$-graph $\Lambda$ with no sources we
let $\Gg_{\Lambda}$  denote the \emph{infinite-path groupoid} of $\Lambda$ introduced in
\cite[Definition 2.7]{KumPas}.

\begin{cor}\label{thm.ip}
Let $\Lambda$ be a row-finite $k$-graph with no sources. Then the following are
equivalent:
\begin{itemize}
\item[(i)] The $k$-graph $\Lambda$ is strongly aperiodic.
\item[(ii)] For every hereditary saturated subset $H\subsetneq\Lambda^0$, the
    groupoid $\Gg_{\Lambda\setminus \Lambda H}$ is topologically principal.
\item[(iii)] The groupoid $\Gg_{\Lambda}$ is essentially principal.
\item[(iv)] The $C^*$-algebra $C^*( \Lambda)$ has topological dimension zero.
\item[(v)] Every ideal of $C^*(\Lambda)$ is gauge invariant.
\item[(vi)] The map $I \mapsto \{v \in \Lambda^0 : s_v \in I\}$ is a bijection
    between ideals of $C^*(\Lambda)$ and hereditary saturated subsets
    $H\subseteq\Lambda^0$.
\end{itemize}
\end{cor}
\begin{proof}
See \cite[Proposition~4.5]{KumPas} for (i)$\Leftrightarrow$(ii). For
(ii)$\Leftrightarrow$(iii), one checks that the nonempty closed invariant subsets of
$\Gg_{\Lambda}^{(0)}$ are precisely the sets $(\Lambda\setminus \Lambda H)^\infty
\subseteq \Lambda^\infty$ associated to hereditary saturated sets $H\subsetneq\Lambda^0$.
This follows from the argument of \cite[Proposition~6.5]{KumPasRaeRen} \emph{mutatis
mutandis}.

Theorem~\ref{thm2.3} gives (i)$\Leftrightarrow$(iv). For (i)$\Leftrightarrow$(v) see
\cite[Proposition~3.6]{RobSim} and see \cite[Theorem~5.5]{Sim} for (v)$\Rightarrow$(vi).
For (vi)$\Rightarrow$(v) observe that the $s_v$ are fixed under the gauge action, so any
ideal generated by vertex projections is gauge invariant.
\end{proof}

\section{Real rank zero}\label{sec:rr0}
In this section we study when $k$-graph $C^*$-algebras have real-rank zero. We can give a
complete answer for $2$-graphs $\Lambda$ such that $C^*(\Lambda)$ is strongly purely
infinite, and we are able to say some things in greater generality. We start with a
necessary condition that follows easily from work of Brown--Pedersen (it could also be
deduced easily from work of Pasnicu \cite{Pas3}) and requires no additional hypotheses.
It also connects the question of real-rank zero to our work in the preceding section.

\begin{lemma}\label{lem:RR0->s.a.}	
Let $\Lambda$ be a row-finite $k$-graph with no sources. If the $C^*$-algebra $C^*(\Lambda)$
has real-rank zero then $\Lambda$ is strongly aperiodic.
\end{lemma}		
\begin{proof}
Suppose that $C^*(\Lambda)$ has real-rank zero. Then it has generalised real-rank zero in
the sense of \cite[2.1(iv)]{BroPed2}, and so Proposition~2.7 of \cite{BroPed2} implies
that it also has topological dimension zero. Hence (i)$\implies$(ii) of
Theorem~\ref{thm2.3} implies that $\Lambda$ is strongly aperiodic.
\end{proof}

Our objective in the rest of the section is to strengthen this necessary condition and
obtain a necessary and sufficient condition for real-rank zero when $C^*(\Lambda)$
is purely infinite and $k= 2$.

For an $R$-module $M$ and an integer $a\geq 2$ let ${\bigwedge}^{a}M$ denote the $a$th
\emph{exterior power} of $M$, i.e. the $R$-module $M^{\otimes a}/J_a$ where $J_a$ is the
submodule of $M^{\otimes a}$ spanned by elements of the form $m_1 \otimes m_2 \otimes
\cdots \otimes m_a$ such that $m_i = m_j$ for some $i \neq j$. For any
$m_1,m_2,\dots,m_a\in M$, the coset of $m_1 \otimes m_2 \otimes \cdots \otimes m_a$ in
${\bigwedge}^{a}M$ is denoted $m_1 \wedge m_2 \wedge \cdots \wedge m_a$ and we call it an
elementary wedge product. If $\{\epsilon_1,\epsilon_2, \dots,\epsilon_N\}$ constitutes a
basis for $M$ (with $a\leq N$) then $\{\epsilon_{i_1} \wedge \epsilon_{i_2} \wedge \cdots
\wedge \epsilon_{i_a} : 1\leq i_1 <\cdots< i_a \leq N \}$ is a basis for
${\bigwedge}^{a}M$. We regard $R$ as an $R$-module and set $\bigwedge^0M = R$ and
$\bigwedge^1M=M$.

Let $\Lambda$ be a row finite $k$-graph with no sources. Following \cite{KumPasSim} we
introduce the chain complex $D_*^\Lambda$. Let $\ZZ \Lambda^0$ denote the set of finitely
supported functions from $\Lambda^0$ to $\ZZ$, regarded as an abelian group under
pointwise addition. For $v\in \Lambda^0$ define $\delta_v\in \ZZ \Lambda^0$ by
$\delta_v(u)=\delta_{u,v}$ (the Kronecker delta). For each $i=1,\dots,k$ let $M_i$ (or
$M_{i,\Lambda}$ for emphasis) be the \emph{vertex connectivity matrix} given by
$M_i(u,v)=|u\Lambda^{e_i}v|$. Regard $M_i$ as the group endomorphism of $\ZZ \Lambda^0$
given by $(M_if)(u)=\sum_{v\in \Lambda^0}M_i(u,v)f(v)$. Let $D^\Lambda_*$ be the chain
complex such that
$$D_a^\Lambda=\begin{cases}
	\bigwedge^a \ZZ^k \otimes \ZZ \Lambda^0 & \mbox{ if } 0\leq a\leq k, \\
	0 & \mbox{ if }a>k,
\end{cases}$$
with differentials $\partial_a : D^\Lambda_a \to D^\Lambda_{a-1}$ (or
$\partial_a^\Lambda$ for emphasis) defined for $1\leq a\leq k$ by
\begin{align*}
\partial_a(\epsilon_{i_1} \wedge \cdots \wedge \epsilon_{i_a} \otimes \delta_v)=\sum_{j = 1}^a (-1)^{j+1} \epsilon_{i_1}\wedge \cdots \wedge\widehat{\epsilon}_{i_j}\wedge\cdots\wedge \epsilon_{i_a} \otimes (1 - M_j^t)\delta_v,
\end{align*}
where the symbol `` $\widehat{\cdot}$ '' denotes deletion of an element, $v\in\Lambda^0$,
$\{\epsilon_1,\dots,\epsilon_k\}$ is the canonical basis for $\ZZ^k$ and $M_j^t$ is the
transpose of $M_j$.

\begin{eg}\label{eq4.1}
Let $\Lambda$ be a row finite $k$-graph with no sources, $H\subsetneq \Lambda^0$ a
hereditary saturated subset of $\Lambda^0$ and $\Gamma={H \Lambda}$. Recall that a
\emph{morphism of chain complexes} $f:C_* \to D_*$ is given by a family of morphisms
$f_n:C_n \to D_n$ that commute with the differentials. In this example we construct a
morphism of chain complexes $\qH : D^\Gamma_* \to D^\Lambda_*$.

For each $f\in \ZZ\Gamma^0$ let $\ext{f}$ be the extension of $f$ to $\Lambda^{0}$ by
zero. Since $H$ is hereditary, for $u\in H$ and $v\not\in H$ we have
$M_{j,\Lambda}(u,v)=0$. It follows that for each $v\in \Lambda^0$
\begin{align*}
(M_{j,\Lambda}^t(\ext{f}))(v)&=\sum_{u\in\Lambda^0}M_{j,\Lambda}(u,v)(\ext{f}(u))=\sum_{u\in \Gamma^0}M_{j,\Lambda}(u,v)(f(u))\\
&= 1_{\Gamma^0}(v) \sum_{u\in \Gamma^0}M_{j,\Lambda}(u,v)(f(u))
=(\ext{M_{j,\Gamma}^tf})(v).
\end{align*}
We conclude that $\ext{(1 - M_{j,\Gamma}^t)f}= (1 - M_{j,\Lambda}^t)(\ext{f})$ for each
$f\in \ZZ \Gamma^0$. For $0\leq a\leq k$ let $\qHa{a} : D^\Gamma_a \to D^\Lambda_a$ be
the morphism given by $\qHa{a}(d \otimes f)=d \otimes \ext{f}$. The above calculation
produces the commuting diagram
\[
\xymatrix{0 \ar[r] & D^\Gamma_k \ar[d]^{\qHa{k}} \ar[r]^-{\partial^\Gamma_k} & D^\Gamma_{k-1} \ar[d]^{\qHa{k-1}} \ar[r] & \cdots \ar[r] & D^\Gamma_1 \ar[d]^{\qHa{1}} \ar[r]^-{\partial^\Gamma_1} & \ar[d]^{\qHa{0}} D^\Gamma_0 \ar[r] & 0\\
0 \ar[r] & D^\Lambda_k \ar[r]^-{\partial^\Lambda_k} & D^\Lambda_{k-1} \ar[r] & \cdots \ar[r] & D^\Lambda_1 \ar[r]^-{\partial^\Lambda_1} & D^\Lambda_0 \ar[r] & 0}
\]
and hence yields a morphism of chain complexes $\qH : D^\Gamma_* \to D^\Lambda_*$.
\end{eg}
	Any morphism of chain complexes induces a morphism of the homology groups of the
chain complexes. In particular Example~\ref{eq4.1} yields the morphism $H_*(\qH) :
H_*(D^\Gamma_*) \to H_*(D^\Lambda_*)$. We can now state our main result.

\begin{thm}\label{thm4.2}
Let $\Lambda$ be a row-finite $2$-graph with no sources.
\begin{itemize}
\item[(1)] Suppose $C^*(\Lambda)$ is purely infinite. Then the following are equivalent:
\begin{itemize}
\item[(i)] $C^*(\Lambda)$ has real rank zero
\item[(ii)] $\Lambda$ is strongly aperiodic and $H_1(\qH)$ is injective for every
    saturated hereditary subset $H$ of $\Lambda^0$.
\end{itemize}
\item[(2)] The implication $\emph{(i)}\Rightarrow\emph{(ii)}$ holds even if $C^*(\Lambda)$ is not purely infinite.
\end{itemize}
\end{thm}

We will prove Theorem~\ref{thm4.2} at the end of the section, after establishing some
preliminary results. First though, we present a reformulation of the condition that
$H_1(\qH)$ is injective in terms of the connectivity matrices of the $2$-graph $\Lambda$.
This result will prove useful when applying Theorem~\ref{thm4.2} in practice. To state
it, we need a little bit more notation. Suppose that $\Lambda$ is a $2$-graph and $H
\subseteq \Lambda^0$ is a hereditary set. It follows that, for each $i$,
$M_{i,\Lambda}(u,v) = 0$ whenever $u \in H$ and $v \not\in H$. This means that each
$M^t_{i,\Lambda}$ has a block-upper-triangular decomposition with respect to the
decomposition $\Lambda^0 = H \sqcup (\Lambda^0 \setminus H)$. We will use the following
notation for this decomposition:
\begin{equation}\label{eq:block decomp}
M^t_{i,\Lambda} =
    \left(
        \begin{array}{cc}
            M^t_{i, H} & M^t_{i, H, \Lambda^0 \setminus H} \\
            0 & M^t_{i, \Lambda^0 \setminus H}
        \end{array}
    \right).
\end{equation}

\begin{prop}\label{injective.equival}
Let $\Lambda$ be a row-finite $2$-graph with no sources, and let $H \subseteq \Lambda^0$
be a saturated hereditary set. With $H_1(\qH)$, $\partial_1$,  $\partial_2$ and $\Gamma$
as in Example~\ref{eq4.1}, the following are equivalent:
\begin{itemize}
    \item[(i)] The morphism $H_1(\qH)$ is injective.
    \item[(ii)] We have $\im \partial^\Lambda_2 \cap \qHa{1}(\ker
    \partial^\Gamma_1) \subseteq \qHa{1}(\im \partial^\Gamma_2 )$.
    \item[(iii)] We have
\[
    \binom{M^t_{1, H, \Lambda^0 \setminus H}}{M^t_{2, H, \Lambda^0 \setminus H}}\big(\ker (M^t_{1, \Lambda^0 \setminus H}-1) \cap \ker (M^t_{2, \Lambda^0
        \setminus H}-1)\big) \subseteq \binom{M^t_{1, H}-1}{M^t_{2, H}-1} \ZZ H.
    \]
\end{itemize}
\end{prop}
\begin{proof}
Before proving (i)$\Leftrightarrow$(ii), we do some initial work to reformulate
condition~(i). Since $\qHa{1}\circ
\partial^\Gamma_2=\partial^\Lambda_2\circ \qHa{2}$, we have $\qHa{1}(\im
\partial^\Lambda_2)\subseteq \im \partial^\Lambda_2$. Hence there is a well-defined
homomorphism $H_1(\qH) : \ker \partial^\Gamma_1 / \im \partial^\Gamma_2 \to \ker
\partial^\Lambda_1 / \im \partial^\Lambda_2$ such that
\[
H_1(\qH)\big(a+\im \partial^\Gamma_2\big) = \qHa{1}(a)+\im \partial^\Lambda_2
    \quad\text{ for all $a \in \ker \partial^\Gamma_1$.}
\]
By linearity, it follows that $H_1(\qH)$ is injective if and only if whenever $a\in \ker
\partial^\Gamma_1$ satisfies $\qHa{1}(a)\in \im \partial^\Lambda_2$, we have $a\in\im \partial^\Gamma_2$.

Now, to prove (i)$\Rightarrow$(ii), suppose that $H_1(\qH)$ is injective, and fix $a\in
\ker \partial^\Gamma_1$ such that $\qHa{1}(a) \in \im\partial^\Lambda_2$. Then the
preceding paragraph shows that $a\in\im \partial^\Gamma_2$, giving $\qHa{1}(a)\in
\qHa{1}(\im \partial^\Gamma_2)$.

To verify (ii)$\Rightarrow$(i), we suppose that~(ii) holds, and fix $a\in \ker
\partial^\Gamma_1$ with $\qHa{1}(a)\in \im \partial^\Lambda_2$. By the first paragraph of
the proof, it suffices to show that $a \in \im \partial^\Gamma_2$. We have $\qHa{1}(a)\in
\im \partial^\Lambda_2 \cap \qHa{1}(\ker \partial^\Gamma_1)$. By~(ii), $\qHa{1}(a)\in
\qHa{1}(\im \partial^\Gamma_2 )$, so $\qHa{1}(a)=\qHa{1}(c)$ for some $c\in\im
\partial^\Gamma_2$. Since $\qHa{1}(d \otimes f)=d \otimes \ext{f}$, where $\ext{f}$ is
the extension of $f\in \ZZ\Gamma^{0}$ to $\Lambda^{0}$ by zero, $\qHa{1}$ is injective,
so $a=c\in\im
\partial^\Gamma_2$.

We now prove that (ii)$\Leftrightarrow$(iii). By \cite{Eva, KumPasSim}, the complex
$D_a^\Lambda=\bigwedge^a \ZZ^k \otimes \ZZ \Lambda^0$ may be written as follows:
\[
\xymatrix{0 \ar[r] & \ZZ \Lambda^0 \ar[r]^-{\partial^\Lambda_2} & {\ZZ \Lambda^0\oplus \ZZ \Lambda^0} \ar[r]^-{\partial^\Lambda_1} & \ZZ \Lambda^0 \ar[r] & 0}
\]
where $\partial^\Lambda_1 = (1 - M_{1,\Lambda}^t, 1 - M_{2,\Lambda}^t)$ and
$\partial^\Lambda_2 =\begin{pmatrix} M_{2,\Lambda}^t - 1 \\
1 - M_{1,\Lambda}^t\end{pmatrix}$. The same applies to the $2$-graph $\Gamma=H\Lambda$.
To ease notation in the calculations that follow, we write $T := \Lambda^0 \setminus H$,
and we write $A := M^t_{1,\Lambda}$ and $B := M^t_{2,\Lambda}$, so that
\[
1 - M^t_{1,\Lambda} =
    \left(
        \begin{array}{cc}
            1 - A_H & -A_{H, T} \\
            0 & 1 - A_T
        \end{array}
    \right)\qquad\text{ and }\qquad
1 - M^t_{2,\Lambda} =
    \left(
        \begin{array}{cc}
            1 - B_H & -B_{H, T} \\
            0 & 1 - B_T
        \end{array}
    \right) .
\]
Define
\[
K_1 :=\im \partial_2^\Lambda,  \qquad
    K_2 :=\qHa{1}(\ker \partial_1^\Gamma),\qquad\text{and}\qquad
    L_1 :=\qHa{1}(\im \partial_2^\Gamma),
\]
so that~(ii) is satisfied if and only if $K_1\cap K_2 \subseteq L_1$. We have
\begin{align*}
K_1 &=
\bigg\{\bigg(\Big(\begin{matrix}
                        B_H - 1 & B_{H, T} \\
                        {0} & B_T - 1\end{matrix}\Big)
                        \Big(\begin{matrix}a\\b\end{matrix}\Big) ,
             \Big(\begin{matrix}
                        1 - A_H & -A_{H, T} \\
                        {0} & 1 - A_T\end{matrix}\Big)
                        \Big(\begin{matrix}a\\b\end{matrix}\Big)\bigg)
            : \Big(\begin{matrix}a\\b\end{matrix}\Big) \in \ZZ\Lambda^0\bigg\},\\
K_2 &= \qHa{1}(\{(u, v) : u,v\in \ZZ{H} \textrm{ and } A_H u + B_H v = u+v\})\\
    &=\Big\{\Big(\minivector{u}{0}, \minivector{v}{0}\Big) : u,v \in \ZZ{H} \textrm{ and } A_H u + B_H v = u+v\Big\},\text{ and}\\
L_1 &= \qHa{1}\Big(\big\{\big((B_H-1) a, (1-A_H) a\big) : a\in \ZZ{H}\big\}\Big)\\
    &=\Big\{\Big(\minivector{(B_H-1) a}{0}, \minivector{(1-A_H) a}{0}\Big) : {a}\in \ZZ{H}\Big\}.
\end{align*}
Suppose that $\Big(\minivector{u}{0}, \minivector{v}{0}\Big) \in K_1 \cap K_2$. The
description of $K_1$ shows that $\Big(\minivector{u}{0}, \minivector{v}{0}\Big)$ has the
form $\partial_2^\Lambda\minivector{a}{b}$ where $b \in \ker(1-B_T) \cap \ker(1-A_T)$.
Conversely, if $b \in \ker(1-B_T) \cap \ker(1-A_T)$ and $a \in \ZZ H$, then we have
\[
\partial_2^\Lambda\minivector{a}{b}
    = \Big(\minivector{(B_H-1) a + B_{H, T} b}{0},
            \minivector{(1-A_H) a - A_{H, T} b}{0}\Big).
\]
For $a \in \ZZ \Lambda^0$, we write $a_H$ for the component $(a_v)_{v \in H}$ of $a$
belonging to $\ZZ H$. Using that $M^t_{1,\Lambda}$ and $M^t_{2,\Lambda}$ commute, we have
\begin{align*}
A_H\big((B_H-1) a + B_{H, T} b\big) +{}& B_H\big((1-A_H) a - A_{H, T} b\big) \\
    &= \Big(M^t_{1,\Lambda}(M^t_{2,\Lambda} - 1)\minivector{a}{b}\Big)_H
        + \Big(M^t_{2,\Lambda}(1 - M^t_{1,\Lambda})\minivector{a}{b}\Big)_H \\
    &= \Big((M^t_{2,\Lambda} - 1)\minivector{a}{b}\Big)_H
        + \Big((1 - M^t_{1,\Lambda})\minivector{a}{b}\Big)_H \\
    &= \big((B_H - 1)a + B_{H, T} b\big)
        + \big((1 - A_H)a - A_{H, T} b\big).
\end{align*}
Thus $\partial_2^\Lambda\minivector{a}{b}$ belongs to $K_1 \cap K_2$. That is, we have
showed that
\[
K_1 \cap K_2
    =\{\partial_2^\Lambda\minivector{a}{b} : a \in \ZZ H\text{ and }
        b \in \ker(1 - B_T) \cap \ker(1 - A_T)\}.
\]
Let
\[
L_2:= \{\left(\minivector{B_{H, T}b}{0}, \minivector{-A_{H, T}b}{0}\right) : {b}\in
\ZZ T\textrm{ and } b \in \ker(1 - B_T) \cap \ker(1 - A_T)\}.
\]
We have shown that $K_1\cap K_2 = L_1+L_2$, and hence~(ii) is satisfied if and only if
$L_2\subseteq L_1$. As the pairs of vectors in $L_1$ and $L_2$ have zeros at the second
row, it follows that~(ii) holds if and only if
\begin{align*}
\{(B_{H, T}b, -A_{H, T}b) : b \in \ZZ T & \textrm{ and } b \in \ker(1 - B_T) \cap \ker(1 - A_T)\} \\
    &\subseteq \big\{\big((B_H-1)a, (1-A_H)a\big) : a\in \ZZ{H}\big\}.
\end{align*}
Multiplying by $-1$ in the second coordinate and recalling the definitions of the
matrices $A$ and $B$ gives (ii)$\Leftrightarrow$(iii).
\end{proof}

We now embark on the preliminary results needed for the proof of Theorem~\ref{thm4.2}.
Let $\Lambda$ be a row finite $k$-graph with no sources. As usual, given an action
$\alpha$ of a locally compact group $G$ on a $C^*$-algebra $A$, we write $i_G : C^*(G)
\to \mathcal{M}(A \times_\alpha G)$ and $i_A : A \to \mathcal{M}(A \times_\alpha G)$ for
the canonical inclusions into the multiplier algebra of the crossed product.

We will abbreviate the crossed product $C^*(\Lambda)\times_\gamma \TT^k$ of
$C^*(\Lambda)$ by its gauge action by $B^\Lambda$. For each $n\in \ZZ^k$ and $f\in
C_c(\TT^k,C^*(\Lambda))$ define $\alpha^\Lambda_{n}(f)(z):=z^{-n}f(z)$. Then
$\alpha^\Lambda_{n}$ extends to an automorphism of $B^\Lambda$ and produces a dynamical
system $(B^\Lambda,\ZZ^k,\alpha^\Lambda)$. The action $\alpha^\Lambda$ is called the
\emph{dual action}, see \cite[p.~190]{Wil}.

Takai duality implies that $C^*(\Lambda)$ is isomorphic to a full corner of $B^\Lambda
\times_{\alpha^{\Lambda}} \ZZ^k$. Specifically, there is an inclusion map $\theta_\Lambda
: C^*(\Lambda) \to B^\Lambda \times_{\alpha^{\Lambda}} \ZZ^k$, which is given by
$\theta_\Lambda(s_\lambda) :=
i_{B^\Lambda}(i_{\TT}(1)i_{C^*(\Lambda)}(s_\lambda))i_{\ZZ^k}(d(\lambda))^*$, and which
induces the identity map in $K$-theory\footnote{This is discussed in detail in terms of
skew-product graphs in \cite[Lemma~5.2]{KumPasSim2}; to translate back to the language of
crossed-products, recall that there is an isomorphism $C^*(\Lambda) \times_\gamma \TT^k
\cong C^*(\Lambda \times_d \ZZ^k)$ that carries each $i_{C^*(\Lambda)}(s_\lambda)
i_{\TT^k}(z^n)$ to $s_{(\lambda, n)}$ \cite[Corollary~5.3]{KumPas}.}.

\begin{lemma}\label{lem4.3}
Let $\Lambda$ be a row finite $k$-graph with no sources, $H\subsetneq \Lambda^0$ a
hereditary saturated subset of $\Lambda^0$ and $\Gamma={H \Lambda}$. There is a
homomorphism $\varphi : C^*(\Gamma)\to C^*(\Lambda)$ satisfying
\begin{align}\label{eqn4.2}
\varphi(s_\lambda)=s_\lambda\quad\text{ for each }\quad \lambda\in \Gamma.
\end{align}
This $\varphi$ extends to a homomorphism $\hat{\varphi} : B^\Gamma \times_{\alpha^\Gamma}
\ZZ^k \to B^\Lambda \times_{\alpha^\Lambda} \ZZ^k$ such that the following diagram
commutes:
\begin{equation}\label{eqn4.3}
\xymatrix{C^*(\Gamma) \ar[d]^{\varphi}\ar[r]^-{\theta_\Gamma} & \ar[d]^{\hat\varphi} B^\Gamma \times_{\alpha^\Gamma} \ZZ^k\\
C^*(\Lambda) \ar[r]^-{\theta_\Lambda} &  B^\Lambda \times_{\alpha^\Lambda} \ZZ^k.}
\end{equation}
\end{lemma}
\begin{proof}
Lemma~\ref{lem2.6} gives a homomorphism $\varphi : C^*(\Gamma)\to C^*(\Lambda)$
satisfying \eqref{eqn4.2}. This $\varphi$ is equivariant for the gauge actions on
$C^*(\Gamma)$ and $C^*(\Lambda)$, so it induces a homomorphism $\varphi\times \id :
B^\Gamma\to B^\Lambda$ mapping $C_c(\TT^k,C^*(\Gamma))$ to $C_c(\TT^k,C^*(\Lambda))$ via
$\varphi\times \id(f)(z)=\varphi(f(z))$, see \cite[Corollary~2.48]{Wil}.

By construction, we have $(\varphi \times\id)\circ i_{C^*(\Gamma)}(s_\lambda) =
i_{C^*(\Lambda)}\circ \varphi(s_\lambda)$ for every $\lambda \in \Gamma$. The commutativity
of the diagram
\[
\label{diagram}
\xymatrix{C_c(\TT^k, C^*(\Gamma)) \ar[r]^-{\varphi\times \id} \ar[d]^{\alpha_n^\Gamma} & \ar[d]^{\alpha_n^\Lambda} C_c(\TT^k, C^*(\Lambda))\\
C_c(\TT^k, C^*(\Gamma)) \ar[r]^-{\varphi\times \id} & C_c(\TT^k, C^*(\Lambda)),}
\]
shows that $\varphi\times \id$ is equivariant for the dual actions of $\ZZ^k$ on
$B^\Gamma$ and $B^\Lambda$. Hence $\varphi\times \id$ induces a homomorphism
$\hat{\varphi} : B^\Gamma \times_{\alpha^\Gamma} \ZZ^k \to B^\Lambda
\times_{\alpha^\Lambda} \ZZ^k$. Using the description of $\theta_\Lambda$ and
$\theta_\Gamma$ above, we see that the diagram~\eqref{eqn4.3} commutes.
\end{proof}

Following \cite{KumPasSim} a \emph{homology spectral sequence} $(E^r,d^r,\alpha^r)$
or simply $(E^r,d^r)$ consists of the following data:
\begin{itemize}
\item[(i)] A family of $\ZZ$-modules $E^r_{a,b}$ defined for all integers $a,b$ and
    $r>0$.
\item[(ii)] Maps $d^r_{a,b} : E^r_{a,b}\to E^r_{a-r,b+r-1}$  that are differentials,
    i.e, $d^r_{a-r,b+r-1}\circ d^r_{a,b}=0$
\item[(iii)] Maps $\alpha^r_{a,b} : E^{r+1}_{a,b}\to H(E^r_{a,b})$ that are
    isomorphisms, where $H(E^{r}_{a,b})$ denotes the module
    $\ker(d^r_{a,b})/\im(d^r_{a+r,b-r+1})$.
\end{itemize}
A \emph{morphism of spectral sequences} $f:(E^r,d^r)\to ({\Etilde}^r,{\dtilde}^r)$ is a
family of morphisms $f^r:E^r\to{\Etilde}^r$ such that (1) the morphisms $f^r$ commute
with the differentials, and (2) the induced homomorphisms $H(f^r_{a,b}) : H(E^r_{a,b})\to
H({\Etilde}^r_{a,b})$ make the following diagram commute:
\[
\xymatrix{E^{r+1}_{a,b} \ar[d]^{f^{r+1}_{a,b}}\ar[r]^-{\alpha^r_{a,b}} & \ar[d]^{H(f^r_{a,b})} H(E^r_{a,b})\\
{\Etilde}^{r+1}_{a,b} \ar[r]^-{{\alphatilde}^r_{a,b}} &  H({\Etilde}^r_{a,b}).}
\]

Recall that the spectral sequence $E^r_{a,b}$ is \emph{bounded} if for each $r$ and $n$
there are just finitely many nonzero terms $E^r_{a,b}$ such that $a+b = n$ (note that
this holds for all $r$ if it holds for $r = 1$). We then eventually have $E^{r+1}_{a,b}
\cong E^r_{a,b}$ for all $a,b$, and we write $E^\infty_{a,b}$ for this limiting term. We
say that the spectral sequence $E^r_{a,b}$ \emph{converges} to a sequence $\mathcal{K}_*
= \{\Kk_n : n \in \ZZ\}$ of modules if each $\Kk_n$ admits a finite filtration $0 =
F_s(\Kk_n) \subseteq F_{s+1}(\Kk_n) \subseteq \cdots \subseteq F_t(\Kk_n) = \Kk_n$ such
that $E^\infty_{a,b} \cong F_a(\Kk_{a+b})/F_{a-1}(\Kk_{a+b})$ for $s < a \le t$.

The following result essentially follows from the argument of \cite[Theorem~3.15]{Eva}
(though we are following the notation of \cite[Proposition~5.4 and
Theorem~5.5]{KumPasSim}), but we need a more detailed statement.
\begin{lemma}\label{lem4.4}
Let $\Lambda$, $\Gamma$ and $\hat{\varphi}$ be as in Lemma~\ref{lem4.3}. There exist
spectral sequences $(E^r,d^r)$ and $({\Etilde}^r,{\dtilde}^r)$ and a morphism of spectral
sequences $f : (E^r,d^r)\to ({\Etilde}^r,{\dtilde}^r)$ such that
\begin{itemize}
\item[(a)] $(E^r,d^r)$ converges to $K_*(B^\Gamma \times_{\alpha^\Gamma} \ZZ^k)$ and
	$({\Etilde}^r,{\dtilde}^r)$ converges to
    $K_*(B^\Lambda \times_{\alpha^\Lambda} \ZZ^k)$,
\item[(b)] $E^1_{a,b}=D_a^\Gamma$ and ${\Etilde}^1_{a,b}=D_a^\Lambda$ for $0 \leq a
    \leq k$ and $b$ even,
\item[(c)] $d^1_{a,b}=\partial_a^\Gamma$ and ${\dtilde}^1_{a,b}=\partial_a^\Lambda$
    for $1 \leq a \leq k$ and $b$ even,
\item[(d)] the isomorphisms $E^\infty_{a,b} \cong K_*(B^\Gamma \times_{\alpha^\Gamma}
    \ZZ^k)$ and $\Etilde^\infty_{a,b} \cong K_*(B^\Lambda \times_{\alpha^\Lambda}
    \ZZ^k)$ of part~(a) intertwine the morphisms $f^\infty_{a,b} : E^\infty_{a,b} \to
    \Etilde^\infty_{a,b}$ and the homomorphisms $K_*(\hat\varphi) : K_*(B^\Gamma
    \times_{\alpha^\Gamma} \ZZ^k)\to K_*(B^\Lambda \times_{\alpha^\Lambda} \ZZ^k)$,
    and
\item[(e)] $f^1_{a,b} : D_a^\Gamma \to D_a^\Lambda$ coincides with the map $\qHa{a}$
    of Example~\ref{eq4.1}.
\end{itemize}	
\end{lemma}
\begin{proof}
Firstly we briefly recall parts of the proof of \cite[Theorem~5.5]{KumPasSim} that we
will need. Since the vertex projections $s_v \in C^*(\Lambda)$ are fixed by the gauge
action, the map that sends $a \in C^*(\Lambda)$ to the constant function $z \mapsto a \in
C(\TT^k, C^*(\Lambda))$ determines a homomorphism $\varepsilon^\Lambda : C_0(\Lambda^0)
\cong C^*(\{s_v : v \in \Lambda^0\}) \to B^\Lambda$; so $\varepsilon^\Lambda(\delta_v)$
is the constant function $z \mapsto s_v$ on $\TT^k$. Let $\varepsilon^\Lambda_* :
\ZZ\Lambda^0 \to K_0(B^\Lambda)$ be the induced map in $K_0$. The map
$$\textstyle{\id\otimes \varepsilon^\Lambda_* : {\bigwedge}^*\,\ZZ^k \otimes \ZZ \Lambda^0 \to {\bigwedge}^*\,\ZZ^k \otimes K_0(B^\Lambda)}$$
is a map of complexes that induces an isomorphism on homology \cite[Theorem~3.14]{Eva}.
Thus, as in \cite{Eva, KumPasSim},
\begin{equation}\label{eqn4.3a}
\textstyle{H_*(\ZZ^k,K_0(B^\Lambda))\cong H_*({\bigwedge}^*\ZZ^k \otimes \ZZ \Lambda^0).}
\end{equation}
Therefore setting ${\Etilde}^1_{a,b}=D^\Lambda_a$ for $0\leq a\leq k$ and $b$ even, and
${\dtilde}^1_{a,b}=\partial^\Lambda_a$ for $1\leq a\leq k$ and $b$ even (and zero
otherwise), yields
\begin{align*}
{\Etilde}^2_{a,b}\cong\left\{
\begin{array}{ll}
H_a(\ZZ^k,K_0(B^\Lambda)) &\mbox{if } 0\leq a\leq k \mbox{ and }b\mbox{ even}, \\
0&\mbox{otherwise}.
\end{array}\right.
\end{align*}
It follows that the spectral sequence $({\Etilde}^r,{\dtilde}^r)$ converges to
$K_*(B^\Lambda \times_{\alpha^\Lambda} \ZZ^k)$ as required.

We now diverge slightly from the proof of \cite[Theorem~5.5]{KumPasSim} and apply
Lemma~\ref{lem4.3}: Fix a hereditary saturated set $H\subsetneq \Lambda^0$, and let
$\Gamma={H \Lambda}$. Repeating the previous argument on $\Gamma$ (using $E$ rather than
$\Etilde$ for the new sequence) we get spectral sequences $(E^r,d^r)$,
$({\Etilde}^r,{\dtilde}^r)$ satisfying properties (a)--(c).

Define $\varepsilon^\Gamma_*$ in the same way as $\varepsilon^\Lambda_*$ and let
$$\iota : C_0(\Gamma^0)\to C_0(\Lambda^0)$$ be the extension map extending by zero.
Using Lemma~\ref{lem4.3} and the canonical identifications $K_0(C_0(\Gamma^0))\cong
\ZZ\Gamma^0$ and $K_0(C_0(\Lambda^0))\cong \ZZ\Lambda^0$, we have the commuting diagram
below:
\begin{equation}\label{eqn4.3b}
\xymatrix{\ZZ\Gamma^0 \ar[rr]^-{\varepsilon^\Gamma_*} \ar[d]^{K_0(\iota)} && \ar[d]^{K_0(\varphi\times \id)} K_0(B^\Gamma)\\
\ZZ\Lambda^0 \ar[rr]^-{\varepsilon^\Lambda_*} && K_0(B^\Lambda).}
\end{equation}
Taking the tensor product of the chain complex $\bigwedge^* \ZZ^k$ with the groups
appearing in~\eqref{eqn4.3b} and applying the homology functor, we get
\[
\xymatrix{H_*(D^\Gamma_*)=H_*(\bigwedge^* \ZZ^k \otimes \ZZ\Gamma^0) \ar[rr]^-{H_*(\id\otimes \varepsilon^\Gamma_*)} \ar[d]^{H_*(\id\otimes K_0(\iota))} && \ar[d]^{H_*(\id\otimes K_0(\varphi\times \id))} H_*(\bigwedge^* \ZZ^k \otimes K_0(B^\Gamma))\\
H_*(D^\Lambda_*)=H_*(\bigwedge^* \ZZ^k \otimes \ZZ\Lambda^0) \ar[rr]^-{H_*(\id\otimes \varepsilon^\Lambda_*)} && H_*(\bigwedge^* \ZZ^k\otimes K_0(B^\Lambda)).}
\]

As in the displayed equation at the top of \cite[p.~187]{KumPasSim} we have isomorphisms
$H_*(\bigwedge^* \ZZ^k \otimes K_0(B^\Gamma))\to H_*(\ZZ^k, K_0(B^\Gamma))$ and
$H_*(\bigwedge^* \ZZ^k \otimes K_0(B^\Lambda))\to H_*(\ZZ^k, K_0(B^\Lambda))$ such that
the following diagram commutes
\begin{equation}\label{eqn4.4}
\xymatrix{H_*(D^\Gamma_*) \ar[rr]^-{H_*(\id\otimes \varepsilon^\Gamma_*)} \ar[d]^{H_*(\id\otimes K_0(\iota))} && \ar[d]^{H_*(\id\otimes K_0(\varphi\times \id))} H_*(\bigwedge^* \ZZ^k \otimes K_0(B^\Gamma)) \ar[r] & H_*(\ZZ^k, K_0(B^\Gamma)) \ar[d]\\
H_*(D^\Lambda_*) \ar[rr]^-{H_*(\id\otimes \varepsilon^\Lambda_*)} && H_*(\bigwedge^* \ZZ^k\otimes K_0(B^\Lambda))\ar[r] & H_*(\ZZ^k, K_0(B^\Lambda)).}
\end{equation}
The composition of the horizontal maps in the diagram \eqref{eqn4.4} is the isomorphism
described by Evans in \cite[Theorem~3.14]{Eva}. Hence $H_*(\id\otimes
\varepsilon^\Gamma_*)$ and $H_*(\id\otimes \varepsilon^\Lambda_*)$ are isomorphisms. Thus
the left-hand vertical map on homology (induced by $\id\otimes K_0(\iota)$) coincides
with the right-hand map (induced by $\id\otimes K_0(\varphi\times \id)$).

Applying the naturality of \cite[Proposition~5.4]{KumPasSim} yields a morphism $f$ of
spectral sequences compatible with $K_*(\hat\varphi) : K_*(B^\Gamma
\times_{\alpha^\Gamma} \ZZ^k)\to K_*(B^\Lambda \times_{\alpha^\Lambda} \ZZ^k)$ such that
$f^1_{a,b} : D_a^\Gamma \to D_a^\Lambda$ is given by $\id\otimes K_0(\iota)$. The latter
map agrees with $\qHa{a}$ from Example~\ref{eq4.1} giving (d)--(e) as required.
\end{proof}
Recall that a $C^*$-algebra $A$ is \emph{$K_0$-liftable} if for every pair of ideals
$I\subseteq J$ in $A$, the extension
\[
\xymatrix{0 \ar[r] & I \ar[r]^\varphi & J \ar[r]^\pi & J/I \ar[r] & 0}
\]
has the property that the induced map $K_0(\pi): K_0(J) \to K_0(J/I)$ is surjective, or
equivalently that $K_1(\varphi): K_1(I) \to K_1(J)$ is injective (see
\cite[Definition~3.1]{PasRor}). In fact, it suffices to consider $J = A$, a fact that the
authors of \cite{PasRor} attribute to Larry Brown. So $A$ is $K_0$-liftable if
$K_1(\varphi): K_1(I) \to K_1(A)$ is injective for every ideal $I$ in $A$.

The next lemma allows us to relate $K_0$-liftability of $C^*(\Lambda)$ to the homology of
the complex $D^\Lambda$ described earlier.

\begin{lemma}\label{lem:liftability<->homology}
Let $\Lambda$ be a row-finite $2$-graph with no sources. Suppose that $H \subseteq
\Lambda^0$ is saturated and hereditary. Let $\qH : D_*^\Gamma \to D_*^\Lambda$ be the map
of complexes induced by the inclusion $\Gamma^0 = H \hookrightarrow \Lambda^0$, and let
$\varphi : C^*(\Gamma) \to C^*(\Lambda)$ be the inclusion that carries a generator
$s_\lambda$ of $C^*(\Gamma)$ to the corresponding generator $s_\lambda$ of
$C^*(\Lambda)$. Then there are isomorphisms $K_1(C^*(\Gamma)) \cong H_1(D^\Gamma)$ and
$K_1(C^*(\Lambda)) \cong H_1(D^\Lambda)$ making the following diagram commute:
\[
\xymatrix{K_1(C^*(\Gamma)) \ar[r]^{\cong} \ar[d]^{K_1(\varphi)} & H_1(D^\Gamma_*) \ar[d]^{H_1(\qH)}\\
K_1(C^*(\Lambda)) \ar[r]^{\cong} & H_1(D^\Lambda_*).}
\]
In particular, $K_1(\varphi)$ is injective if and only if $H_1(\qH)$ is injective.
\end{lemma}
\begin{proof}
Define $\Gamma={H \Lambda}$. By Lemma~\ref{lem4.4} we have the spectral sequence
$({\Etilde}^r,{\dtilde}^r)$ converging to $K_0(B^\Lambda \times_{\alpha^\Lambda} \ZZ^2)$.
Since $\Lambda$ is a $2$-graph and by the convergence from Lemma~\ref{lem4.4}(a) we have
the following filtration of $\Kktilde_1=K_1(B^\Lambda \times_{\alpha^\Lambda} \ZZ^2)$:
\[
0=F_0(\Kktilde_1)\subseteq F_1(\Kktilde_1)\subseteq F_2(\Kktilde_1)=\Kktilde_1,
\]
with $F_1(\Kktilde_1)\cong {\Etilde}^2_{1,0}$, $F_2(\Kktilde_1)\cong \Kktilde_1$ and
$F_2(\Kktilde_1)/F_1(\Kktilde_1) = \{0\}$ (see \cite{Eva}). The same applies to the
convergent spectral sequence $({E}^r,{d}^r)$. We let ${\Kk}_*=K_*(B^\Gamma
\times_{\alpha^\Gamma} \ZZ^2)$ denote its limit. Since $f$ is compatible with
$K_*(\hat\varphi) : {\Kk}_*\to {\Kktilde}_*$ by Lemma~\ref{lem4.4}(d), we have the
commuting diagram
\begin{equation}\label{diag4.3}
\xymatrix{{\Kk}_1 \ar[r] \ar[d]^{K_0(\hat\varphi)} & E^2_{1,0} \ar[d]^{f^2_{1,0}}\\
{\Kktilde}_1 \ar[r] & {\Etilde}^2_{1,0},}
\end{equation}
and the horizontal maps are isomorphisms. Using that ${\Etilde}^2_{1,0}\cong
H({\Etilde}^1_{1,0})=H_1(D^\Lambda_*)$ and $E^2_{1,0}\cong H(E^1_{1,0})=H_1(D^\Gamma_*)$
we obtain the following commuting diagram:
\[
\xymatrix{E^2_{1,0} \ar[d]^{f^2_{1,0}}\ar[rr]^-{\alpha^1_{1,0}} && \ar[d]^{H_1(\qH)} H_1(D^\Gamma_*)\\
{\Etilde}^2_{1,0} \ar[rr]^-{{\alphatilde}^1_{1,0}} &&  H_1(D^\Lambda_*).}
\]
Applying the $K_1$-functor to the commuting diagram
\[
\xymatrix{C^*(\Gamma) \ar[d]^{\varphi}\ar[r] & \ar[d]^{\hat\varphi} B^\Gamma \times_{\alpha^\Gamma} \ZZ^2\\
C^*(\Lambda) \ar[r] &  B^\Lambda \times_{\alpha^\Lambda} \ZZ^2}
\]
of Lemma~\ref{lem4.3} we get a commuting diagram in which the horizontal maps are
invertible. Combining this and that $\Kktilde_1=K_1(B^\Lambda \times_{\alpha^\Lambda}
\ZZ^2)$ and $\Kk_1=K_1(B^\Gamma \times_{\alpha^\Gamma} \ZZ^2)$ with the previous diagrams
we obtain the commuting diagram
\[
\xymatrix{K_1(C^*(\Gamma)) \ar[r] \ar[d]^{K_1(\varphi)} & \ar[r] \ar[d]^{K_0(\hat\varphi)} K_1(B^\Gamma \times_{\alpha^\Gamma} \ZZ^2)=\Kk_1 & \ar[r] E^2_{1,0} \ar[d]^{f^2_{1,0}} & H_1(D^\Gamma_*) \ar[d]^{H_1(\qH)}\\
K_1(C^*(\Lambda)) \ar[r] & \ar[r] K_1(B^\Lambda \times_{\alpha^\Lambda} \ZZ^2)=\Kktilde_1 & \ar[r] {\Etilde}^2_{1,0} & H_1(D^\Lambda_*),}
\]
where the horizontal maps are all isomorphisms. The final statement follows immediately.
\end{proof}

\begin{proof}[Proof of Theorem~\ref{thm4.2}]
We start by proving (2), so assume that $C^*(\Lambda)$ has real-rank zero.
Lemma~\ref{lem:RR0->s.a.} implies that $\Lambda$ is strongly aperiodic. As alluded to
immediately after \cite[Definition~3.1]{PasRor}, all $C^*$-algebras of real rank zero are
$K_0$-liftable: Theorem~2.10 of \cite{BroPed} shows that if $A$ has real-rank zero, then
so does every $M_n(A)$, and then applying \cite[Theorem~3.14]{BroPed} to each $M_n(A)$
shows that if $I$ is an ideal of $A$ then every projection in $M_n(A/I)$ lifts to a
projection in $M_n(A)$. Since $A$, and hence also $A/I$, has real-rank zero, $K_0(A/I)$
is generated by classes of projections over $A/I$, and it follows that $A$ is
$K_0$-liftable. Hence every inclusion $I\subseteq C^*(\Lambda)$ of a proper ideal induces
an injective map on $K_1$-groups. In particular, for every hereditary saturated subset
$H\subseteq \Lambda^0$, the map $H_1(\qH)$ is injective by
Lemma~\ref{lem:liftability<->homology}.

We now prove~(1). Suppose that $C^*(\Lambda)$ is purely infinite. The preceding paragraph
establishes (i)$\Rightarrow$(ii). To prove the reverse implication (ii)$\Rightarrow$(i)
suppose that $\Lambda$ is strongly aperiodic and that $H_1(\qH)$ is injective for each
$H$. Theorem~\ref{thm2.3} implies that $C^*(\Lambda)$ has topological dimension zero. So
\cite[Theorem~4.2]{PasRor} implies that $C^*(\Lambda)$ has real-rank zero if and only if
it is $K_0$-liftable, and it suffices to establish the latter.

Fix a non-trivial ideal $I$ of $C^*(\Lambda)$. We need to show that the inclusion
$\iota\colon I \hookrightarrow C^*(\Lambda)$ induces an injective map at the level of
$K_1$-groups. Since $\Lambda$ is strongly aperiodic it follows from
Corollary~\ref{thm.ip} that $I=I_H$ for some nonempty hereditary saturated $H\subsetneq
\Lambda^0$. Let $\Gamma := H\Lambda$. As in Lemma~\ref{lem2.6}, the $C^*$-algebra
$C^*(\Gamma)$ is isomorphic to a full corner of $I$, and so the inclusion map
$C^*(\Gamma) \hookrightarrow I$ induces an isomorphism in $K$-theory. We obtain the
following commuting diagram:
\[
\xymatrix{K_1(C^*(\Gamma)) \ar[r] \ar[d]^{K_1(\varphi)} & K_1(I) \ar[d]^{K_1(\iota)}\\
K_1(C^*(\Lambda)) \ar[r]^{\id} & K_1(C^*(\Lambda)).}
\]
So $K_1(\iota)$ is injective if and only if $K_1(\varphi)$ is injective. By
condition~(ii), the map $H_1(\qH)$ is injective, and so
Lemma~\ref{lem:liftability<->homology} shows that $K_1(\varphi)$ is injective as
required.
\end{proof}

To finish off the section, we show how to reduce checking real-rank zero for 2-graph
algebras with finite ideal lattice to checking real-rank zero for simple 2-graph
algebras, without any assumption of pure infiniteness. We thank an anonymous referee for
pointing out this application of our arguments.

\begin{prop}
Let $\Lambda$ be a row-finite $2$-graph with no sources, and suppose that $C^*(\Lambda)$
has finite ideal lattice. Then $\Lambda$ is strongly aperiodic, and the following are
equivalent:
\begin{itemize}
\item[(1)] $C^*(\Lambda)$ has real-rank zero;
\item[(2)] Both of the following hold:
    \begin{itemize}
    \item[(i)] for every saturated hereditary $H \subseteq \Lambda^0$, the map
        $H_1(\qH)$ of Example~\ref{eq4.1} is injective; and
    \item[(ii)] for every pair $K \subseteq H$ where $K$ is a saturated
        hereditary subset of $\Lambda^0$ and $H \setminus K$ is a saturated
        hereditary subset of $\Lambda \setminus \Lambda K$ such that $H \Lambda
        \setminus \Lambda K$ is cofinal, the $C^*$-algebra $C^*(H \Lambda
        \setminus \Lambda K)$ has real-rank zero.
    \end{itemize}
\end{itemize}
\end{prop}
\begin{proof}
Every $C^*$-algebra with finite ideal lattice has topological dimension zero, and so
Theorem~\ref{thm2.3} shows that $\Lambda$ is strongly aperiodic.

First suppose that $C^*(\Lambda)$ has real-rank zero. Then part~(2) of
Theorem~\ref{thm4.2} gives~(2i). Fix a nested pair $K \subseteq H$ of saturated
hereditary sets such that $H \Lambda \setminus \Lambda K$ is cofinal. Then $C^*(\Lambda
\setminus \Lambda H)$ is a quotient of $C^*(\Lambda)$ by \cite[Theorem~5.2(b)]{RaeSimYee}
and therefore has real-rank zero. Moreover, $C^*(H \Lambda \setminus \Lambda K)$ is
isomorphic to a full corner of an ideal in $C^*(\Lambda \setminus \Lambda K)$ by
\cite[Theorem~5.2(c)]{RaeSimYee}, and therefore itself also has real-rank zero,
giving~(2ii).

Now suppose that~(2) holds. We prove that $C^*(\Lambda)$ has real-rank zero by induction
on $|\Prim C^*(\Lambda)|$. If $\Prim C^*(\Lambda)$ is a singleton, then $\Lambda$ itself
is cofinal \cite[Proposition~4.8]{KumPas}, so we can take $K = \emptyset$ and $H =
\Lambda^0$ in~(2) to see that $C^*(\Lambda) = C^*(H\Lambda \setminus \Lambda K)$ has
real-rank zero by hypothesis. Now fix $\Lambda$ and suppose that our assertion holds for
every strongly aperiodic $k$-graph with fewer primitive ideals than $C^*(\Lambda)$. Since
$\Prim C^*(\Lambda)$ is finite, we can fix an ideal $I$ of $C^*(\Lambda)$ such that
$C^*(\Lambda)/I$ is simple. Since $\Lambda$ is strongly aperiodic, $I$ is generated by
$\{p_v : v \in K\}$ for some saturated hereditary $K \subseteq \Lambda^0$. We then have
$C^*(\Lambda)/I \cong C^*(\Lambda \setminus \Lambda K)$ by
\cite[Theorem~5.2(b)]{RaeSimYee} again, and since this $C^*$-algebra is simple, $\Lambda
\setminus \Lambda K$ is cofinal. So~(2ii) with $H = \Lambda^0$ says that $C^*(\Lambda)/I$
has real-rank zero. Lemma~\ref{lem:liftability<->homology} and~(2i) show that
$C^*(\Lambda)$ is $K_0$-liftable, and it follows from \cite[Proposition~3.15]{BroPed}
that if $I$ has real-rank zero, then every projection in $C^*(\Lambda)/I$ lifts to a
projection in $C^*(\Lambda)$. Therefore, by \cite[Theorem~3.14]{BroPed}, to prove that
$C^*(\Lambda)$ has real-rank zero, it now suffices to prove that $I$ has real-rank zero.

Consider the $k$-graph $\Gamma := K\Lambda$. By Lemma~\ref{lem2.6}, $C^*(\Gamma)$ is
isomorphic to a full corner of $I$, so \cite[Theorem~3.8]{BroPed} implies that it
suffices to prove that $C^*(\Gamma)$ has real-rank zero. Since $C^*(\Lambda)$ is
$K_0$-liftable, so is $I$, and hence also $C^*(\Gamma)$ since the inclusion map
$C^*(\Gamma) \hookrightarrow I$ induces an isomorphism in $K$-theory. The $k$-graph
$\Gamma$ is strongly aperiodic because $\Lambda$ is strongly aperiodic. Suppose that $K'$
and $H'$ are saturated hereditary subsets of $\Gamma^0$ as in~(2ii) for $\Gamma$; so $H'
\setminus K'$ is saturated and hereditary in $\Gamma \setminus \Gamma K'$, and $H'\Gamma
\setminus\Gamma K'$ is cofinal. Then $H'$ and $K'$ also satisfy these hypotheses with
respect to $\Lambda$ (because $\Gamma^0 = K$ is saturated in $\Lambda$).  Hence
hypothesis~(2ii) for $C^*(\Lambda)$ implies that $C^*(H' \Gamma \setminus \Gamma K') =
C^*(H'\Lambda \setminus \Lambda K')$ has real-rank zero. Thus $\Gamma$ satisfies
(2i)~and~(2ii). By construction, $C^*(\Gamma) = C^*(\Lambda)/I$ has fewer primitive
ideals than $C^*(\Lambda)$, and so the inductive hypothesis implies that $C^*(\Gamma)$
has real-rank zero as required.
\end{proof}

\section{Purely infinite \texorpdfstring{$k$}{k}-graph \texorpdfstring{$C^*$}{C*}-algebras}\label{sec:pi}
In this section we turn our attention to pure infiniteness of higher-rank graph
$C^*$-algebras, which is one of the assumptions in Theorem~\ref{thm4.2}. In
\cite{KR-infty}, Kirchberg and R{\o}rdam introduced three separate notions of purely
infinite $C^*$-algebras; weakly purely infinite, purely infinite and strongly purely
infinite. As the names suggest, strong pure infiniteness implies pure infiniteness which
implies weak pure infiniteness. Of these notions, strong pure infiniteness is perhaps the
most useful in the classification of non-simple $C^*$-algebras. Indeed, Kirchberg showed
in \cite{Kir04} that two separable, nuclear, stable, strongly purely infinite
$C^*$-algebras with the same primitive ideal space $X$ are isomorphic if and only if they
are $KK_X$-equivalent.

Following \cite{KirRor}, we denote the set of positive elements in a $C^*$-algebra $A$ by
$A^+$. The ideal in $A$ generated by an element $b$ is denoted $\overline{AbA}$. Recall
that for positive elements $a\in M_n(A)$ and $b\in M_m(A)$, we say that $a$ is {\em Cuntz
below} $b$, denoted $a\precsim b$, if there exists a sequence of elements $x_k$ in
$M_{m,n}(A)$ such that $x_k^* b x_k\to a$ in norm. We say $A$ is {\em purely infinite} if
there are no characters on $A$ and for all $a,b\in A^+$, we have $a\precsim b$ if and
only if $a\in \overline{AbA}$ (see \cite[Definition~4.1]{KirRor}).

Let $\Lambda$ be a $2$-graph and $u \in \Lambda^0$. As introduced in \cite{KanPas}, given
$a,b \in \ZZ^+$, an $(a,b)$-\emph{aperiodic quartet}, or just an \emph{aperiodic
quartet}, at $u$ consists of distinct paths $\alpha_1, \alpha_2 \in u \Lambda^{a e_1} u$
and distinct paths $\beta_1, \beta_2 \in u \Lambda^{b e_2} u$ such that $\beta_2 \alpha_1
= \alpha_1 \beta_2$, $\beta_2 \alpha_2 = \alpha_2 \beta_2$, $\beta_1 \alpha_1 = \alpha_2
\beta_1$, and $\beta_1 \alpha_2 = \alpha_1 \beta_1$.

\begin{prop} \label{prop:aq-pis}
Let $\Lambda$ be a row-finite $2$-graph with no sources. Suppose that there is an
aperiodic quartet at $u$ for each $u\in \Lambda^0$. Then $C^*(\Lambda)$ is strongly
purely infinite.
\end{prop}

Before giving a proof of Proposition~\ref{prop:aq-pis} we present a characterisation of
when the $C^*$-algebra $C^* ( \Lambda )$ of a row-finite strongly aperiodic $k$-graph
$\Lambda$ with no sources is purely infinite. Recall that a projection $p$ in a
$C^*$-algebra $A$ is {\em infinite} if it is Murray--von Neumann equivalent to a proper
subprojection of itself, and {\em properly infinite} if there are two mutually orthogonal
subprojections of $p$ in $A$, each one Murray--von Neumann equivalent to $p$.

\begin{lemma}\label{lem5.2}
Let $\Lambda$ be a row-finite $k$-graph with no sources. Suppose that $\Lambda$ is
strongly aperiodic. Then the following are equivalent
\begin{itemize}
\item[(i)] The $C^*$-algebra $C^*(\Lambda)$ is purely infinite.
\item[(ii)] For every $v\in \Lambda^0$ the projection $s_v$ is properly infinite in
    $C^*( \Lambda)$.
\item[(iii)] For every hereditary saturated $H\subseteq \Lambda^0$ and every $v\in
    \Lambda^0\setminus H$, the projection $s_v$ is infinite in $C^*(\Lambda\setminus
    \Lambda H)$.
\item[(iii')] For every hereditary saturated $H\subseteq \Lambda^0$ and every $v\in
    \Lambda^0\setminus H$, the projection $s_v$ is properly infinite in
    $C^*(\Lambda\setminus \Lambda H)$.
\item[(iv)] For every hereditary saturated $H\subseteq \Lambda^0$ and every $v\in
    \Lambda^0\setminus H$, the $C^*$-algebra $s_vC^*(\Lambda\setminus \Lambda H)s_v$
    contains an infinite projection.
\item[(iv')] For every hereditary saturated $H\subseteq \Lambda^0$ and every $v\in
    \Lambda^0\setminus H$, the $C^*$-algebra $s_vC^*(\Lambda\setminus \Lambda H)s_v$
    contains a properly infinite projection.
\item[(v)] Every non-zero hereditary sub-$C^*$-algebra in any quotient of $C^*(
    \Lambda)$ contains an infinite projection.
\item[(v')] Every non-zero hereditary sub-$C^*$-algebra in any quotient of $C^*(
    \Lambda)$ contains a properly infinite projection.
\end{itemize}
\end{lemma}
\begin{proof}[Proof of Lemma~\ref{lem5.2}]
We first prove
(i)$\Leftrightarrow$(ii)$\Leftrightarrow$(iii)$\Leftrightarrow$(iv)$\Leftrightarrow$(v).

The implication (i)$\Rightarrow$(ii) is known, see \cite[Theorem~4.16]{KirRor}. For
(ii)$\Rightarrow$(iii), observe that \cite[Theorem~5.2]{RaeSimYee} shows that
$C^*(\Lambda\setminus \Lambda H)$ is isomorphic to a quotient of $C^*(\Lambda)$, and that
the quotient map carries $s_v$ to zero if and only if $v \in H$. Corollary~3.15 of
\cite{KirRor} says that the image of a properly infinite projection under a
$C^*$-homomorphism is either zero or infinite, giving (ii)$\Rightarrow$(iii). The
implication (iii)$\Rightarrow$(iv) is trivial because $s_v\in s_vC^*(\Lambda\setminus
\Lambda H)s_v$ for each $v \in \Lambda^0 \setminus H$.

For (iv)$\Rightarrow$(v), fix a proper ideal $I$ of $C^*(\Lambda)$. Let $B$ be a non-zero
hereditary subalgebra of $C^*( \Lambda)/I$. By Corollary~\ref{thm.ip} and
\cite[Theorem~5.2]{RaeSimYee}, there is a saturated hereditary subset $H \subseteq
\Lambda^0$ such that $\Gamma := \Lambda \setminus \Lambda H$ satisfies
$C^*(\Lambda)/I\cong C^*(\Gamma)$. So $B$ is isomorphic to a hereditary subalgebra of
$C^*(\Gamma)$. We identify the two henceforth.

Select any non-zero positive element $a\in B$. Since $\Lambda$ is strongly aperiodic,
Theorem~\ref{thm.ip} implies that $\Gg_\Gamma$ is topologically principal. Hence
\cite[Lemma~3.2]{BroClaSie} provides a non-zero positive element $h\in C^*(\{s_\lambda
s^*_\lambda : \lambda \in \Gamma\})\cong C_0(\Gg^{(0)}_{\Gamma})$ such that $h\precsim
a$. Recall that the cylinder sets $\{Z(\lambda) : \lambda \in \Gamma\}$ form a basis for
the topology on $\Gg^{(0)}_\Lambda$ (see \cite[Definition 2.4]{KumPas}). Choose
$\lambda\in \Gamma$ such that $h(x)\geq \|h\|/2>0$ for all $x\in Z(\lambda)$, and let $v
:= s(\lambda) \in \Gamma^0$. Let $g\in C_0(\Gg^{(0)}_{\Gamma})$ be the function
$g(x):=h(x)^{-1/2}$ for $x\in Z(\lambda)$ and zero otherwise. Then
$ghg=1_{Z(\lambda)}=s_\lambda s_\lambda^*$ and
$(gs_\lambda)^*hgs_\lambda=s_\lambda^*s_\lambda s_\lambda^*s_\lambda=s_{v}$. In
particular, $s_{v}\precsim h$ (take $x_k := gs_\lambda$ for each $k$). By assumption~(iv)
there is an infinite projection $p\in s_vC^*(\Gamma)s_v$. By
\cite[Proposition~2.7]{KirRor} we have $p\precsim s_v$. Hence $p\precsim s_v\precsim h
\precsim a$, so $p\precsim a$. As explained after \cite[Proposition~2.6]{KirRor}, there
exists $x\in C^*(\Gamma)$ such that $p=x^*ax$. Let $y := a^{1/2}x$. Then $y^*y=x^*ax=p$,
and $q:=yy^*=a^{1/2}xx^*a^{1/2}\in B$. Since $p$ is infinite and Murray--von Neumann
equivalent to $q$, we conclude that $q \in B$ is infinite.

The implication (v)$\Rightarrow$(i) follows from \cite[Proposition~4.7]{KirRor}.

It now suffices to prove
(ii)$\Rightarrow$(iii')$\Rightarrow$(iv')$\Rightarrow$(v')$\Rightarrow$(v). Recall that
if $p$ is properly infinite and $\phi$ is a $C^*$-homomorphism such that $\phi(p) \not=
0$, then $\phi(p)$ is properly infinite. This and the second statement of
Lemma~\ref{lem2.6} gives (ii)$\Rightarrow$(iii'). The implication
(iii')$\Rightarrow$(iv') is trivial, and (iv')$\Rightarrow$(v') follows from the argument
for (iv)$\Rightarrow$(v) above because any projection $p$ that is Murray--von Neumann
equivalent to a properly infinite projection is itself is properly infinite. Since proper
infiniteness is stronger than infiniteness, condition~(v') implies~(v), completing the
proof.
\end{proof}

\begin{cor}	
Let $\Lambda$ be a row-finite $k$-graph with no sources and suppose that $C^*(\Lambda)$
has real-rank zero. Then $C^*(\Lambda)$ is strongly purely infinite if
and only if $s_v$ is properly
infinite for every $v\in \Lambda^0$.
\end{cor}
\begin{proof}
This follows from \cite[Corollary~6.9]{PasRor} combined with and Lemmas
\ref{lem:RR0->s.a.}~and~\ref{lem5.2}.
\end{proof}

Following \cite{PasRor}, recall that a $C^*$-algebra $A$ has the \emph{ideal property},
abbreviated (IP), if projections in $A$ separate ideals in $A$; that is, whenever $I$,
$J$ are distinct ideals in $A$, there is a projection in $I\setminus J$.

\begin{lemma}\label{lem5.3}
Let $\Lambda$ be a row-finite $k$-graph with no sources. Suppose that $\Lambda$ is
strongly aperiodic. Then $C^*(\Lambda)$  has the ideal property.
\end{lemma}
\begin{proof}
Let $I,J$ be ideals of $C^*( \Lambda)$ such that $I\not\subseteq J$. Recall that for $H
\subseteq \Lambda^0$, $I_H$ denotes the ideal in $C^*(\Lambda)$ generated by $\{s_w :
w\in H\}$. By Corollary~\ref{thm.ip}, $I=I_H$ and $J=I_K$ for some saturated hereditary
$H,K\subseteq \Lambda^0$. Since $I\not\subseteq J$, we have $H\not\subseteq K$, say $v
\in H \setminus K$. \cite[Theorem~5.2(b)]{RaeSimYee} shows that $H = \{w : s_w \in I_H\}$
and similarly for $K$, so we deduce that $s_v$ is a projection in $I_H \setminus I_K = I
\setminus J$.
\end{proof}	

\begin{rmk}\label{rem5.4}
It not clear whether strong aperiodicity of $\Lambda$ is equivalent to property~(IP) for
$C^*(\Lambda)$. However, we obtain some easy partial results. First, suppose
$C^*(\Lambda)$ is AF---so automatically has property~(IP). By
\cite[Proposition~3.12]{EvaSim}, every ideal of $C^*( \Lambda)$ is gauge-invariant, so
Corollary~\ref{thm.ip}, shows that $\Lambda$ is strongly aperiodic. Second, suppose that
$C^*( \Lambda)$ has~(IP) and is purely infinite. By \cite[Proposition~2.11]{PasRor},
$C^*(\Lambda)$ has topological dimension zero, so Corollary~\ref{thm.ip} shows that
$\Lambda$ is strongly aperiodic.
\end{rmk}

\begin{proof}[Proof of Proposition~\ref{prop:aq-pis}]
Since every vertex of $\Lambda$ has an aperiodic quartet, \cite[Proposition~3.9]{KanPas}
shows that $\Lambda$ is strongly aperiodic. Hence Lemma~\ref{lem5.3} shows that $C^*(\Lambda)$ has
property~(IP). By \cite[Proposition~2.14]{PasRor}, a $C^*$-algebra with property~(IP) is
purely infinite if and only if it is strongly purely infinite. So it suffices to prove
that $C^*(\Lambda)$ is purely infinite.

To show that $C^*(\Lambda)$ is purely
infinite it suffices to verify property~(ii) of Lemma~\ref{lem5.2}; that is, it suffices
to show that every $s_v$ is properly infinite.

Fix any $v\in \Lambda^0$. By assumption there exist $a,b \in \ZZ^+$, distinct
$\alpha_1,\alpha_2 \in v \Lambda^{a e_1} v$ and distinct $\beta_1, \beta_2 \in v
\Lambda^{b e_2} v$ such that $\beta_2 \alpha_1 = \alpha_1 \beta_2$, $\beta_2 \alpha_2 =
\alpha_2 \beta_2$, $\beta_1 \alpha_1 = \alpha_2 \beta_1$, and $\beta_1 \alpha_2 =
\alpha_1 \beta_1$. By (CK3)--(CK4), for $i = 1,2$ we have
\[
s_v =s_{\alpha_i}^* s_{\alpha_i},
    \qquad s_{\alpha_i} s_{\alpha_i}^*\leq \sum_{\lambda \in v\Gamma^{a {e_1}}} s_\lambda s_\lambda^*=s_v.
\]
Consequently, there exist distinct mutually orthogonal subprojections of $s_v$ in $C^*(
\Lambda)$, each Murray--von Neumann equivalent to $s_v$. We conclude that $s_v$ is
properly infinite.
\end{proof}

\section{Examples}\label{sec:examples}
In this section we present three examples highlighting the necessity of the hypotheses in
Theorem~\ref{thm4.2}(1). All three examples are constructed using 2-graphs. Our result
says that the combination of strong aperiodicity of $\Lambda$, injectivity of each
$H_1(\qH)$, and pure infiniteness of $C^*(\Lambda)$ is sufficient to guarantee that
$C^*(\Lambda)$ has real rank zero. Our three examples show that no combination of two of
these conditions is strong enough.

We describe our examples using the \emph{2-coloured graphs} of \cite{HazRaeSimWeb}. A
$2$-coloured graph is a directed graph endowed with a map $c  : E^1 \to \{c_1,c_2\}$. We
think of $c$ as determining a \emph{colour map} from $E^*$ to the free abelian semigroup
$\mathbb{F}_2$ generated by $\{c_1, c_2\}$ and for $w \in \mathbb{F}_2$, we say that
$\lambda \in E^*$ is $w$-coloured if $c(\lambda) = w$. A collection of factorisation
rules for $E$ is a range- and source-preserving bijection $\theta$ from the
$c_1c_2$-coloured paths in $E^*$ to the $c_2c_1$-coloured paths. For $k=2$ the
associativity condition of \cite{HazRaeSimWeb} is trivial, and so
\cite[Theorems~4.4~and~4.5]{HazRaeSimWeb} say that for every $2$-coloured graph $(E,c)$
with a collection $\theta$ of factorisation rules, there is a unique $2$-graph $\Lambda$
with $\Lambda^{e_i} = c^{-1}(c_i)$, $\Lambda^0 = E^0$, and $ef = f'e'$ in $\Lambda$
whenever $\theta(ef) = f'e'$ in $E^*$.

\begin{eg}\label{eq6.3}
Strong aperiodicity of $\Lambda$ (and hence topological dimension zero for
$C^*(\Lambda)$) combined with strong pure infiniteness of $C^*(\Lambda)$ do not suffice
for $C^*(\Lambda)$ to have real-rank zero, even in the special case $k=2$. To see this,
fix $n \ge 3$ and let $(E, c)$ be the following 2-coloured graph:
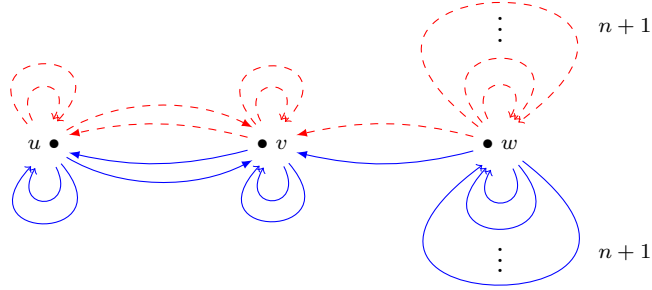
\[
\begin{tikzpicture}[xscale=2, yscale=1.25]
    \node[inner sep=0.5pt, circle] (17) at (-1,8) {\tiny${\ u\ \bullet\ }$};
    \node[inner sep=0.5pt, circle] (07) at (0.5,8) {\tiny${\ \bullet\ v\ }$};
    \node[inner sep=0.5pt, circle] (27) at (2,8) {\tiny${\ \bullet\ w\ }$};
	\node at (2,6.85) {$\vdots$};
	\node at (2.82,6.85) {\tiny{$n+1$}};
	\node at (2,9.3) {$\vdots$};
	\node at (2.82,9.25) {\tiny{$n+1$}};
    \draw[-latex, red, dashed] (07) edge[out=160,in=20] (17);
    \draw[-latex, blue] (07) edge[out=200,in=340] (17);
    \draw[-latex, red, dashed] (17) edge[out=40,in=140] (07);
    \draw[-latex, blue] (17) edge[out=320,in=220] (07);
    \draw[-latex, red, dashed] (27) edge[out=160,in=20] (07);
    \draw[-latex, blue] (27) edge[out=200,in=340] (07);
	\path[->,every loop/.style={looseness=10}] (17)
	         edge  [in=70,out=110,loop, red, dashed] ();
	\path[->,every loop/.style={looseness=14}] (17)
			 edge  [in=60,out=120,loop, red, dashed] ();
	\path[->,every loop/.style={looseness=10}] (17)
	         edge  [in=250,out=290,loop, blue] ();
	\path[->,every loop/.style={looseness=14}] (17)
			 edge  [in=240,out=300,loop, blue] ();	
	\path[->,every loop/.style={looseness=10}] (07)
	         edge  [in=70,out=110,loop, red, dashed] ();
	\path[->,every loop/.style={looseness=14}] (07)
			 edge  [in=60,out=120,loop, red, dashed] ();
	\path[->,every loop/.style={looseness=10}] (07)
	         edge  [in=250,out=290,loop, blue] ();
	\path[->,every loop/.style={looseness=14}] (07)
			 edge  [in=240,out=300,loop, blue] ();

	\path[->,every loop/.style={looseness=10}] (27)
	         edge  [in=70,out=110,loop, red, dashed] ();
	\path[->,every loop/.style={looseness=14}] (27)
			 edge  [in=60,out=120,loop, red, dashed] ();
	\path[->,every loop/.style={looseness=24}] (27)
			 edge  [in=50,out=130,loop, red, dashed] ();			
	\path[->,every loop/.style={looseness=10}] (27)
	         edge  [in=250,out=290,loop, blue] ();
	\path[->,every loop/.style={looseness=14}] (27)
			 edge  [in=240,out=300,loop, blue] ();
	\path[->,every loop/.style={looseness=24}] (27)
			 edge  [in=230,out=310,loop, blue] ();
\end{tikzpicture}
\]
We define factorisation rules as follows. First, for each pair of vertices $x,y$ of $E$,
list the blue edges in $x E^1 y$ as $\{\alpha^{x,y}_1, \dots, \alpha^{x,y}_{|xE^1y|/2}\}$
and the red edges as $\{\beta^{x,y}_1, \dots, \beta^{x,y}_{|xE^1y|/2}\}$. When $x = y$,
we write $\alpha^x_i$ and $\beta^x_i$ instead of $\alpha^{x,x}_i$ and $\beta^{x,x}_i$.
For each vertex $x$, we define four factorisation rules by
\begin{equation}\label{eq:aq}
    \beta^x_2 \alpha^x_1 = \alpha^x_1 \beta^x_2, \quad \beta^x_2 \alpha^x_2 = \alpha^x_2 \beta^x_2, \quad
    \beta^x_1 \alpha^x_1 = \alpha^x_2 \beta^x_1, \quad \beta^x_1 \alpha^x_2 = \alpha^x_1 \beta^x_1.
\end{equation}
We then specify $\alpha^{x,y}_i \beta^{y,z}_j = \beta^{x,y}_i \alpha^{y,z}_j$ for every
blue-red path $\alpha^{x,y}_i \beta^{y,z}_j$ not appearing in the left-hand side of any
of the factorisation rules in~\eqref{eq:aq}. With these factorisation rules we obtain a
$2$-graph $\Lambda$, which is row finite with no sources. The factorisation
rules~\eqref{eq:aq} yield a $\big((1,0), (0,1)\big)$-aperiodic quartet at each vertex $x$
of $\Lambda$. Proposition~\ref{prop:aq-pis} and \cite[Proposition~3.9]{KanPas} ensures
that $C^*(\Lambda)$ is strongly purely infinite and $\Lambda$ is strongly aperiodic for
each $n\geq 1$. It is easy to check that $H:=\{w\}$ is a hereditary saturated subset of
$\Lambda^0$. With the notation of \eqref{eq:block decomp} we have
\[
M^t_{i,H} = (n+1), \qquad M^t_{i,H,\Lambda^0\setminus H} = (0\ 1),
    \quad\text{ and }\quad M^t_{i, \Lambda^0\setminus H}=\minimatrix{2}{1}{1}{2}
\]
for each $i$. With $L_2:=\{M^t_{1,H,\Lambda^0\setminus H}b  : b \in
\ZZ{(\Lambda^0\setminus H)} \textrm{ and } (M^t_{1, \Lambda^0\setminus H}-1)b=0\}$ and
$L_1:=\{(M^t_{1,H}-1)a : a\in \ZZ{H}\}$, we have $L_1 = n\ZZ$ and $L_2 = \ZZ$. So
Proposition~\ref{injective.equival} implies that $H_1(\qH)$ is injective if and only if
$n = 1$. Hence Theorem~\ref{thm4.2}(1) implies that $C^*(\Lambda)$ has real-rank zero for
$n=1$, but not for $n>1$.
\end{eg}

\begin{eg}\label{eq6.4}
The injectivity of each $H_1(\qH)$ combined with strong pure infiniteness of
$C^*(\Lambda)$ do not suffice for $C^*(\Lambda)$ to have real-rank zero, even in the
special case $k=2$. To see this, consider the following 2-coloured graph:
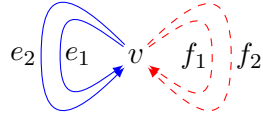
\[
\begin{tikzpicture}[rotate=90]
    \node[inner sep = 1.5pt] (v) at (0,0) {$v$};
    \draw[-latex, blue] (v) .. controls (0.55,0.75) and (0.7,1) .. (0,1)
        node[pos=1,anchor=west,inner sep=1.5pt] {\color{black}\small$e_1$}
        .. controls (-0.7,1) and (-0.5,0.7) .. (v);
    \draw[-latex, blue] (v) .. controls (0.85,0.85) and (1,1.25) .. (0,1.25)
        node[pos=1,anchor=east,inner sep=1.5pt] {\color{black}\small$e_2$}
        .. controls (-1,1.25) and (-0.85,0.85) .. (v);
    \draw[-latex, dashed, red] (v) .. controls (0.5,-0.7) and (0.7,-1) .. (0,-1)
        node[pos=1,anchor=east,inner sep=1.5pt] {\color{black}\small$f_1$}
        .. controls (-0.7,-1) and (-0.55,-0.75) .. (v);
    \draw[-latex, dashed, red] (v) .. controls (0.8,-0.8) and (1,-1.25) .. (0,-1.25)
        node[pos=1,anchor=west,inner sep=1.5pt] {\color{black}\small$f_2$}
        .. controls (-1,-1.25) and (-0.85,-0.85) .. (v);
\end{tikzpicture}
\]
Define factorisation rules on $(E,c)$ by $e_i f_j = f_i e_j$, and let $\Lambda$ be the
resulting $2$-graph. By \cite[Corollary~3.5(iii)]{KumPas}, we have $C^*(\Lambda) \cong
\mathcal{O}_{2} \otimes C(\TT)$. As $\mathcal{O}_{2}$ is purely infinite simple,
$C^*(\Lambda)$ is strongly purely infinite, see \cite[Theorem~1.3]{KirSie2}. There are no
non-trivial saturated hereditary subsets of $\Lambda^0$, so injectivity of each
$H_1(\qH)$ is trivial. The $C^*$-algebra $C^*(\Lambda)$ is non-simple, but none of the
proper ideals of $C^*(\Lambda)$ contain any vertex projections. Consequently, it follows
from Corollary~\ref{thm.ip} that $\Lambda$ is not strongly aperiodic (alternatively, for
each $x\in v\Lambda^{\infty}$, $\sigma^{(1,-1)}(x)=x$). Using Theorem~\ref{thm4.2} we
conclude that $C^*(\Lambda)$ fails to have real-rank zero (cf.~\cite{PasRor}).
\end{eg}

\begin{eg}
Strong aperiodicity combined with the injectivity of each $H_1(\qH)$ does not suffice for
$C^*(\Lambda)$ to have real-rank zero, even in the special case $k=2$. To see this,
consider the $2$-coloured graph:
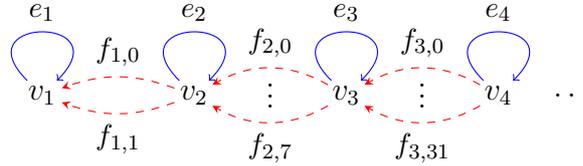
\[
\begin{tikzpicture}
    \node[circle, inner sep=1pt] (v1) at (0,0) {$v_1$};
    \node[circle, inner sep=1pt] (v2) at (2,0) {$v_2$};
    \node[circle, inner sep=1pt] (v3) at (4,0) {$v_3$};
    \node[circle, inner sep=1pt] (v4) at (6,0) {$v_4$};
    \node at (7,0) {\dots};
	\path[->,every loop/.style={looseness=8}] (v1)
	         edge[in=45,out=135,loop, blue]  node[above, pos=0.5,black]{\small$e_1$} ();
	\path[->,every loop/.style={looseness=8}] (v2)
	         edge[in=45,out=135,loop, blue]  node[above, pos=0.5,black]{\small$e_2$} ();
	\path[->,every loop/.style={looseness=8}] (v3)
	         edge[in=45,out=135,loop, blue]  node[above, pos=0.5,black]{\small$e_3$} ();
	\path[->,every loop/.style={looseness=8}] (v4)
	         edge[in=45,out=135,loop, blue]  node[above, pos=0.5,black]{\small$e_4$} ();
    \draw[red, dashed, -stealth, in=20, out=160] (v2) to node[above, pos=0.5, black] {\small$f_{1,0}$} (v1);
    \draw[red, dashed, -stealth, in=340, out=200] (v2) to node[below, pos=0.5, black] {\small$f_{1,1}$} (v1);
    \draw[red, dashed, -stealth, in=30, out=150] (v3) to node[above, pos=0.5, black] {\small$f_{2,0}$} (v2);
    \node at (3,0.1) {$\vdots$};
    \draw[red, dashed, -stealth, in=330, out=210] (v3) to node[below, pos=0.5, black] {\small$f_{2,7}$} (v2);
    \draw[red, dashed, -stealth, in=30, out=150] (v4) to node[above, pos=0.5, black] {\small$f_{3,0}$} (v3);
    \node at (5,0.1) {$\vdots$};
    \draw[red, dashed, -stealth, in=330, out=210] (v4) to node[below, pos=0.5, black] {\small$f_{3,31}$} (v3);
\end{tikzpicture}
\]
that has $2^{2n-1}$ edges labelled $f_{n,0}, \dots, f_{n, 2^{2n-1}-1}$ from $v_{n+1}$ to
$v_n$ for each $n \ge 1$. There is a unique $2$-graph for this $2$-coloured graph with
factorisation rules given by
\[
e_n f_{n, i}
    = \begin{cases}
        f_{n, i+1} e_{n+1} &\text{ if $0 \le i < 2^{n-1}-1$}\\
        f_{n, 0} e_{n+1} &\text{ if $i = 2^{n-1}-1$}\\
        f_{n,i} e_{n+1} &\text{ if $2^{n-1} < i < 2^{2n-1}$;}
    \end{cases}
\]
so factorisation through a blue loop cyclicly permutes the first $2^{n-1}$ edges in $v_n
\Lambda^{e_2} v_{n+1}$, and fixes the remaining ones.

For this example, in the notation of \cite[Theorem~7.2]{PasRaeRorSim}, the $A_n$ are all
nonzero $1 \times 1$ matrices, and each $\mathcal{F}^{1,1}_n$ is the permutation of
$v_{n} \Lambda^{e_2} v_{n+1}$ induced by the factorisation rules. In particular, the
constants $\beta_1$, as in \cite[Definition~7.1]{PasRaeRorSim}, are given by $$\beta_1 =
\frac{1}{2^{2n-1}} \cdot \{i : 2^{n-1} < i < 2^{2n-1}\},$$ and so each $1 -
\beta_1(F_n^{1,1})$ is $2^{n-1}/2^{2n-1} = 2^{-n}$. So $\overline{\alpha}_1 =
\sum_{n=1}^\infty (1 - \beta_1(\mathcal{F}^{1,1}_n)) = \sum 2^{-n}$ converges to $1 <
\infty$. Thus \cite[Theorem~7.2(2)]{PasRaeRorSim} implies that $C^*(\Lambda)$ has
real-rank~1. The graph $\Lambda$ is clearly cofinal. The numbers
$\kappa(\mathcal{F}^{1,1}_n)$ denoting the maximal orders of elements of $v_n
\Lambda^{e_2} v_{n+1}$ under $\mathcal{F}^{1,1}_n$ satisfy $\kappa(\mathcal{F}^{1,1}_n) =
2^{n-1} \to \infty$. Hence the first statement of \cite[Theorem~7.2]{PasRaeRorSim} says
that $C^*(\Lambda)$ is simple. It follows that $\Lambda$ is strongly aperiodic, and each
$H_1(\qH)$ is injective because $H=\Lambda^0$ is the only non-empty saturated hereditary
subset of $\Lambda^0$, in which case $H_1(\qH)$ is the identity map.
\end{eg}

\end{document}